\documentclass[11pt,twoside]{article}
\usepackage[english]{babel}
\usepackage{amssymb}
\usepackage[mathscr]{eucal}
\usepackage{amsmath,amsthm}
\usepackage{mathrsfs}
\usepackage{braket}
\usepackage{lgreek}
\usepackage[shortlabels]{enumitem}
\usepackage[colorlinks=true,
   urlcolor=blue,           filecolor=green,      
   citecolor=green,      
   linkcolor=red,           bookmarks=true,
  unicode,
   plainpages=false,   ]{hyperref}
\usepackage{color}

\definecolor{royalblue}{rgb}{0,0,0.128}
\def\bp{\begin{proof}}
\def\ep{\end{proof}}
\def\n{\nabla}
\def\ssfrac#1#2{\mbox{\large$\frac{#1}{#2}$}}
\def\sfrac#1#2{\mbox{\Large$\frac{#1}{#2}$}}

\def\intll#1#2{\int\limits_{#1}^{#2}}
\def\dm {|\hskip-0.05cm|}
\def\tm{|\hskip-0.05cm|\hskip-0.05cm|}

\def\displ{\displaystyle}
\def\VSE{\vspace{6pt}\\&\displ }
\def\VS{\vspace{6pt}\\\displ }
\def\VSs{\vspace{4pt}\\\displ}
\def\rf#1{{\rm(\ref{#1})}}

\newcommand{\R}{\mathbb{R}}
\newcommand{\N}{\mathbb{N}}

\def\à{à}
\def\vf{\varphi}

\def\dy{\displaystyle}
\def\vep{\varepsilon}

\def\be{\begin{equation}}
\def\ba{\begin{array}}
\def\ea{\end{array}}
\def\ee{\end{equation}}
\def\vs1{\vspace{1ex}}
\def\vp{\varphi}
\def\ov{\overline}

\def\Ã©{\'{e}}
\def\Ãš{\`{e}}
\newtheorem{lemma}
{\bf Lemma} 
\font\sc=cmcsc10
\title{\Large\bf The IBVP for the Navier-Stokes equations in the half-space in a class of weighted Lebesgue spaces}
\author{\sc  Angelica Pia Di Feola and Vittorio Pane\thanks{Dipartimento di Matematica e Fisica,  
Universit\`{a} degli 
Studi della Campania
``L. Vanvitelli'', via Vivaldi 43, 81100 \null\hskip0.55cmCaserta,
 Italy.\newline\null\hskip0.55cm
vittorio.pane2@unicampania.it
 \newline \null\hskip0.55cm 
angelicapia.difeola@unicampania.it}}
\date{\today}
\begin{document}
\markboth{\footnotesize\rm  A. P. Di Feola, V. Pane}{\footnotesize\rm
The IBVP for the Navier-Stokes equations in the half-space in a class of weighted Lebesgue spaces}
\maketitle
\noindent
\newcommand{\red}{\protect\bf}
\renewcommand\refname{\centerline
{\red {\normalsize \bf References}}}
\newtheorem{ass}
{\bf Assumption} 
\newtheorem{defi}
{\bf Definition} 
\newtheorem{theorem}
{\bf Theorem} 
\newtheorem{rem}
{\sc Remark} 
\newtheorem{coro}
{\bf Corollary} 
\newtheorem{prop}
{\bf Proposition} 
\renewcommand{\theequation}{\arabic{equation}}
\setcounter{section}{0}
\section{Introduction}

In the present work, we investigate the initial-boundary value problem (IBVP) for the Navier–Stokes equations in the half-space with initial data belonging to a suitable weighted Lebesgue space. More precisely, denoting by $\mathbb{R}^n_{+}$ the $n$-dimensional upper half-space, with $n\geq 3$, our aim is to study the following problem:
\begin{equation}\label{problem}
\ba{ll}
u_t - \Delta u = - u \cdot \nabla u - \nabla \pi  \,,\; &\text{in}\,\, (0,T) \times \R^n_+\,, \VSs
\nabla \cdot u=0\,, &\text{in}\,\, (0,T) \times \R_+^n \,,\VSs
u = 0\,, &\text{in}\,\, (0,T) \times \{x_n=0\} , \VS
u(0,x)=u_0(x)\,, \; &\text{on}\,\, \{0\} \times \R_+^n\,,
\ea
\end{equation}
under the assumptions that $u_0 \in L^p_{\texttt{w}}(\mathbb{R}^n_{+})$, and $\nabla \cdot u_0=0$ in weak sense. The weight function $\texttt{w}$ and the corresponding weighted Lebesgue space $L^p_{\texttt{w}}(\mathbb{R}^n_{+})$ are defined as follows.\par
Let $p>n$, and $\alpha=1-\frac{n}{p}$. Moreover, for $m\in \N$, let $\{\alpha_j\}_{j=1,\dots,m}$ be such that $\alpha_j\geq 0$ for all $j\in\{1,\dots,m\}$ and $\displaystyle \sum_{j=1}^m\alpha_j=\alpha$. By $\texttt{w}(x)$ we denote the following weight function
\be \label{wf}
\texttt{w}(x)=\prod_{j=1}^m|x-\bar{x}_j|^{\alpha_j}\,,
\ee
where $\bar{x}_j$, for $j=1,\dots,m$, are fixed (not necessarily distinct) points in $\R^n_+$. We consider the following weighted Lebesgue space:
\begin{equation}\label{DSP}
L^p_{\texttt{w}}(\R^n_+):=\{u\, :\, \texttt{w}(x)\,u \in L^p(\R^n_+)\}\,,
\end{equation}
endowed with the natural norm 
$$\dm   u\dm _{L^p_{\texttt{w}}(\R^n_+)}:=\dm u\,\texttt{w}\dm _{L^p(\R^n_+)}\,.$$
Where there is no ambiguity, we will indicate $\dm   \cdot \dm _{L^p_{\texttt{w}}(\R^n_+)} := \dm   \cdot \dm_{\texttt{w},p} $ and $\dm   \cdot \dm _{L^p(\R^n_+)}: = \dm   \cdot \dm _{p}$.\par
The choice of this class of weight functions is motivated by the fact that the corresponding metric is a particular case of a scale-invariant metric. The significance of this scale-invariant structure is reflected in the properties established on the solution and its derivatives, as will become clear below.\par
In the literature, several studies have been devoted to the Navier–Stokes equations in weighted functional spaces. We briefly mention some examples, without attempting to provide an exhaustive list.\par
In \cite{KK0}, after establishing several weighted $L^p$-$L^q$ estimates for the Stokes semigroup in exterior domains and perturbed half-spaces, for $1<p\leq q<\infty$, the authors apply these results to derive temporal asymptotic estimates for solutions to the Navier–Stokes equations. The weight considered therein is of the form $\overline{w}(x)=\langle x\rangle^{sp}=(1+|x|^2)^{\frac{sp}{2}}$, with $1<p<\infty$ and $0\leq s<(n-1)\left(1-\frac{1}{p}\right)$.\par
In the same spirit, in \cite{KK1}, the authors deal with the half-space setting and establish analogous asymptotic-in-time estimates for solutions to the Navier–Stokes equations in weighted spaces associated with a different class of weights. More precisely, the authors consider the space $L^p_{\omega}(\mathbb{R}^n_+)$, where $\omega(x)=\langle x'\rangle^{s_1}\langle x_n\rangle^{s_n}$ for $x=(x',x_n)\in\mathbb{R}^n_+$, with suitable exponents $s_1$ and $s_n$. In this framework, they derive weighted large-time asymptotic estimates for solutions to the Navier–Stokes equations.\par
The use of weighted functions and weighted integrability assumptions is not restricted to the framework considered here and appears in several other problems. For instance, in \cite{farwingarx}, the authors introduce a new approach based on a radially symmetric weight function to study the time-periodic Navier-Stokes problem. Moreover, in \cite{MarPalma}, concerning the interaction problem between a rigid body and a viscous Newtonian fluid, weighted integrability assumption on the gradient of the initial data are imposed in order to establish uniqueness of regular solutions.\par
The present investigation is part of a broader study started in \cite{GM} and later developed in \cite{MP}. In the first work, the authors study perturbations of a stationary state and prove an asymptotic stability result, showing in particular that the pointwise space-time behavior of the perturbations is fully controlled by the size of the initial data. In the second one, the authors consider the Navier--Stokes Cauchy problem and, for perturbations of the rest state, extend the results of \cite{GM} to the $n$-dimensional setting. We note that both of the aforementioned papers deal with the scale-invariant weighted Lebesgue space $L^p_\alpha$ corresponding to the radial weight $|x|^\alpha$, where $\alpha = 1 - \frac{n}{p}$, $p>n\ge3$.\par
The aim of the present work is to generalize the results of \cite{MP} to the half-space setting within a new functional framework, by considering problem \eqref{problem}. Our analysis relies on the results in \cite{DFP}, where the IBVP for the Stokes system\footnote{For further details on this problem in weighted space setting, we refer the reader to \cite{farsohr},\cite{FroII},\cite{KK0}, \cite{KK1} and the references therein.} in the half-space is studied with initial data in $L^p_{\texttt{w}}(\mathbb{R}^n_{+})$, where, with $\texttt{w}(x)$ we indicate the weight function defined in \eqref{wf}. The authors prove existence and uniqueness, establish $L^q$–$L^p_{\texttt{w}}$ semigroup estimates for $q > n$, and derive space-time asymptotic behavior of the solutions (see Theorem\,\ref{Theoremprincipale} and Corollary\,\ref{andamenti_asintotici}).\par
Here, relying on the theory developed in \cite{KS}, \cite{solo}, \cite{soln} and \cite{Solonn} we looking for a mild solution $(u,\pi)$ to problem \eqref{problem} in the form \eqref{Soluzione} and \eqref{Soluzionepres}. This is achieved via a successive approximation method and by introducing a suitable metric defined in \eqref{tremazze}.
  Consequently,  $L^q$-estimates for $q\in(n,\infty]$ and  pointwise space-time estimates for the solution thus constructed are obtained.\par
 Since the initial data are allowed to have arbitrary size, the corresponding existence interval is local in time. Nevertheless, by adapting the method introduced by Crispo and Maremonti in \cite{CM-L3}, and later exploited in \cite{MP}, we establish a qualitative relation between the interval of existence and the properties of the initial data. More precisely, this relation is obtained through the absolute continuity property of the integral. As a consequence, the argument provides only a qualitative description of the dependence of the existence time on the initial data, and does not yield a quantitative estimate of the instant $T$ in terms of their size. \par
In order to state our main result, we recall some definitions. Indicated by $p'$ the conjugate exponent of $p$ and setting $\alpha'_j:=-\alpha_j$,  $\alpha'=-\alpha$, we define
\begin{equation}\label{dwf}
    \texttt{w}'(x)=\prod_{j=1}^m|x-\bar{x}_j|^{\alpha'_j}\,.
\end{equation} 
The corresponding weighted space to the weight function \eqref{dwf}, say $L^{p'}_{\texttt{w}'}(\Omega)$, is the dual space \footnote{The reader is referred to \cite{Kuf} for further details.}  of $L^p_{\texttt{w}}(\Omega)$.\par
Now, we recall the definition of Lorentz spaces\footnote{\,For a wider background, we refer
the reader {\it e.g.} to \cite{Hunt}, \cite{PKJF-LS} .}. Let $\Theta \subseteq \R^n$ measurable, and $g(x)$ be a (Lebesgue) measurable function, we denote by $\mu_g$ the distribution function of $g$, that is:
\[\mu_g(\lambda):=meas(\{x\in \Theta\,:\, |g(x)|>\lambda\})\,,\quad \text{for all}\,\, \lambda>0,\]
and we define the non-increasing rearrangement of $g$ as:
\[g^*(t):=\inf\{\lambda\,:\, \mu_g(\lambda)\leq t\}\,, \quad \text{for}\,\, t\in[0,\infty).\]
Now, assume that $0<p\leq q \leq \infty$. The Lorentz space $L^{p,q}(\Theta)$ is the collection of all measurable functions $g(x)$ such that $\dm g  \dm _{L^{p,q}(\Omega)}<\infty$, where:
\begin{equation*}
\dm g  \dm _{L^{p,q}(\Theta)}=
\begin{cases}
\bigg(\mbox{\Large$ \int$}_0^{\infty}[t^{\frac{1}{p}}g^*(t)]^q\,  \frac{dt}{t}\bigg)^{\frac{1}{q}}\quad &\text{if}\,\, 0<q<\infty,\\
\displaystyle\sup_{t\in (0,\infty)} t^{\frac{1}{p}}g^*(t) \quad &\text{if}\,\, q=\infty.
\end{cases}
\end{equation*}
In particular, as it is well known, for $p\in \left[1, \infty\right]$, $L^{p,p}(\Theta) = L^p(\Theta)$. Moreover, if $p\in \left[1, \infty \right]$ and $0<q_1\leq q_2\leq \infty$, $L^{p,q_1}(\Theta) \subset L^{p,q_2}(\Theta)$. In what follows,  we will indicate $\dm   \cdot \dm _{L^{p,q}(\Theta)} := \dm   \cdot \dm _{(p,q)}\,$.
\par
In the following, with the symbol $X(p)$ we denote the following spaces:
$$X(p):=\left\{\hskip-0.5cm\ba{lll}&L^{p'}_{\texttt{w}'}(\R^n_+)\cap L^{q'}(\R^n_+)\,,q\in(p,2n)\,,&\mbox{for } p\in(n,2n)\,,\VSE L^{p'}_{\texttt{w}'}(\R^n_+)\cap L^{\frac n{n-1},1}(\R^n_+)\,,&\mbox{for }p=2n\,,\VSE L^{p'}_{\texttt{w}'}(\R^n_+)\cap L^{\frac{np}{p(n-2)+2n},1}(\R^n_+)\,,&\mbox{for }p>2n\,.\ea\right.$$\par 
Let $a \in L^p_{\texttt{w}}(\R^n_+)$. For $t>0$ and $\rho>0$, we set
\be \label{krho}
K_a(t,\rho):= c_0 \dm a \dm_{\texttt{w},p,\rho} + c_1 e^{-\frac{\rho^2}{8t}} \dm a \dm_{\texttt{w},p},
\ee
 where
 \be \label{normarho}
\dm a \dm_{\texttt{w},p,\rho}:=\sup_{x\in \R^n_+}\left[\int_{B_\rho(x)} |a(y)|^p \texttt{w}^p(y)\, dy \right]^{\frac{1}{p}},
\ee
and $B_\rho(x)$ is the semisphere of radius $\rho>0$ centered in $x$. \par
We make use of the following metric: for $q>n$, we set
\be \label{tremazze}
\tm u(t)\tm:=t^{\frac n{2p}}\dm u(t)|x|^{\alpha}\dm_\infty+t^\frac12\dm u(t)\dm_\infty+ t^{\frac n2\left(\frac1n-\frac1q\right)}\dm u(t)\dm_q + t^{\frac n2\left(\frac1n-\frac1q \right)+\frac12}\dm \n u(t)\dm_q\,.
\ee \par
In literature, for an initial datum  in belongs to usual $L^p$-spaces, it has been proved the existence of a solution for system \eqref{problem} in the form 
\be 
u(t,x) =\int_{\mathbb{R}^n_+}
\mathcal{G}(t-\tau,x,y)\, u_0 (y) \,dy - \int_0^t \int_{\mathbb{R}^n_+}
\mathcal{G}(t-\tau,x,y)\,P [u \cdot \n u(\tau,y)]\,dy\,d\tau ,\label{Soluzione}
\ee
\small{\be \pi (t,x) \!=\!\int_{\mathbb{R}^n_+}\!\!\!\!\!\!
\mathcal{Q}(t-\tau,x,y)\cdot u_0 (y) \,dy - \int_0^t \! \!\!\int_{\mathbb{R}^n_+}\!\!\!\!\!\!
\mathcal{Q}(t-\tau,x,y)\cdot P [u \cdot \n u(\tau,y)]\,dy\,d\tau +  \int_{\mathbb{R}^n_+}\!\! \!\!\!\!\n_y \mathcal{N}(x,y) \cdot(u\cdot \n u) (t,y) \, dy,\label{Soluzionepres}
\ee}
where $\mathcal{G}$, $\mathcal{Q}$, $ \mathcal{N}$ are defined in \eqref{Gkernel}, \eqref{Q} and \eqref{N}, respectively.\par

The goal of this paper is to establish the following results.

\begin{theorem}\label{esistenza}
{\sl 
Let $u_0 \in L^p_{\texttt{w}}(\R^n_+)$ and divergence free in weak sense. There exists an instant $T:=T(u_0)$ such that problem \eqref{problem} admits a solution $(u,\pi)$, given by \eqref{Soluzione} and \eqref{Soluzionepres}, with $u$ smooth for all $t \in (0,T)$, such that 
  for all $q>n$, 
\begin{equation}\label{andamento_sol_q}
t^\frac n{2p}\dm u(t)|x|^\alpha\dm_\infty+t^{\frac{1}{2}}\dm u(t)\dm _{\infty}+t^{\frac{n}{2}(\frac{1}{n}-\frac{1}{q})}\dm u(t)\dm _q + t^{\frac{n}{2}(\frac{1}{n}-\frac{1}{q})+\frac{1}{2}}\dm \n u(t)\dm _q \leq K(t,\rho)+  c\dm u_0\dm _{\texttt{w},p}^{1-\gamma}K(t,\rho)^{\gamma},
\end{equation}
 for all $t\in(0,T)$, where $\gamma=1-\frac{n}{q}$. 
The following limit properties hold:
\begin{equation}\label{limite_linfinito_soluzione}
\lim_{t \to 0^+} \tm u(t) \tm=0\,, 
\end{equation} \be \label{weak_convergence1}
\lim_{t\to 0^+}\,(u(t),\varphi)=(u_0,\varphi) \,, \text{ for all } \varphi \in  X(p)\,.
\ee
Moreover, concerning the pressure term, for any fixed time $\varepsilon>0$ and for all $r>1$, we obtain
\begin{equation}\label{stime_pressione_thm_principale}
 \dm \n 
\pi_{u } \dm_{L^r(\eta, T; L^q(\R^n_+))} \leq c(r,\varepsilon,\eta, T) \left(   K(\varepsilon,\rho)+  c\dm u_0\dm _{\texttt{w},p}^{1-\gamma}K(\varepsilon,\rho)^{\gamma}\right)\left(  1+ K(\varepsilon,\rho)+  c\dm u_0\dm _{\texttt{w},p}^{1-\gamma}K(\varepsilon,\rho)^{\gamma}\right)\,, 
\end{equation}
for all $\eta >\varepsilon$, where $c$ is a constant independent of $u$.
Finally, if the norm $\dm u_0\dm _{\texttt{w},p}$ is suitably small, then above results hold for all $t>0$.}
\end{theorem}
\begin{rem}\label{R-I}{\rm It is well known that, by standard arguments based on the representation formulas \rf{Soluzione}--\rf{Soluzionepres} (see, e.g., \cite{CMon, MS}), the solution $(u,\pi)$ given by Theorem~\ref{esistenza} enjoys pointwise regularity for all $(t,x)\in (\eta,T)\times \R^n$, for every $\eta>0$.
}
\end{rem}
\begin{rem}
\rm We do not address here a detailed analysis of the pressure term. Nevertheless, we would like to emphasize that, for $p>2n$, our solution belongs to the class of solutions considered in \cite{CMon}. Therefore, for such exponents, the theory developed in the aforementioned work applies to our setting as well.
\end{rem}
\begin{theorem}[Uniqueness]\label{unicita}{\sl
Let $(u,\pi)$ be a solution to problem \eqref{problem} whose existence is ensured by Theorem\,\ref{esistenza}. Then $(u,\pi)$ is unique in the class of existence detected in Theorem\,\ref{esistenza}.}
\end{theorem}
Uniqueness is proved as in \cite{MP}. The argument relies the decomposition of the solution into a linear part and a remainder containing the convective term which allows the uniqueness issue to be reduced to the latter component. The conclusion then follows from a duality argument and a generalized Gronwall inequality (see Lemma\,\ref{L-GWSI}).\par
The paper is organized as follows. In Sect.\,\ref{sect. 2}, we introduce the notation and collect the preliminary results that will be used throughout the paper. In Sect.\,\ref{sect. 3}, we recall the results concerning the Cauchy problem for the Stokes system in the weighted $L^p_{\texttt{w}}$ setting. In Sect.\,\ref{sect. 4}, we introduce the IBVP problem for the Stokes equations in the half space; here, we introduce the Green function related to this problem and present several estimates that are instrumental for our analysis. We then recall the results for the homogeneous Stokes system and subsequently establish pointwise and $L^q$-estimates, with $q \in (n,\infty]$, for the solution to the Stokes initial-boundary value problem with a suitable body force. Finally, in Section\,\ref{sect. 5}, we prove our main results, that is Theorems\,\ref{esistenza} and \ref{unicita}.

\section{Preliminary results}\label{sect. 2}
In this section, we introduce some notations, and further we recall some useful preliminary results.\par
 At first, we remark that the weight function \eqref{wf} is a Muckenhoupt weight, and the space $L^p_{\texttt{w}}(\Omega)$, with the relative norm, can be defined for each measurable set $\Omega \subseteq \R^n$.\par
Hereafter, we use the notation $D_x^{\mathtt{k}}$ for partial derivatives with respect to $x$, where $\mathtt{k}$ is a multi-index with $|\mathtt{k}|=k>1$. For first-order derivatives, i.e., $|\mathtt{k}|=1$, we use the standard notation $\nabla$.
In addition, with the symbol $D_t^h$, we denote  the $h$-th partial derivative with respect to the real variable $t$. \par
For the sake of exposition, all the definitions and results  presented in this section, are given for the half-space $\mathbb{R}^n_+$. The corresponding statements remain valid for the whole space, as well as for bounded and exterior domains, provided suitable regularity assumptions are satisfied. \par
We start by recalling the Helmholtz decomposition in the $L^p_{\texttt{w}}(\R^n_+)$ space. We restrict our attention to the weighted spaces; for a comprehensive overview of the Helmholtz decomposition in the classical Lebesgue spaces and related topics, we refer the reader to \cite{galdilibro}.\par
We define the following spaces\,\footnote{For more details and proof, we refer to \cite{DFP}.}
\begin{gather*}
J^q_{\texttt{w}}(\mathbb{R}_+^n)=\{v\in L^q_{\texttt{w}}(\mathbb{R}_+^n)\, |\, (v,\nabla h)=0 \,\, \text{for all}\, h \in W^{1,q'}_{loc}(\mathbb{R}_+^n,w')\text{,}\, \nabla h \in L^{q'}_{\texttt{w}'}(\R_+^n)\}\,,\\
G^q_{\texttt{w}}(\mathbb{R}_+^n)=\{u\in L^q_{\texttt{w}}(\mathbb{R}_+^n)\, |\, \exists \pi\, : \,\, u=\nabla \pi \text{,}\,\text{with} \, \pi \in W^{1,q}_{loc}(\mathbb{R}_+^n,w)\text{,}\, \nabla \pi \in L^{q}_{\texttt{w}}(\R_+^n)\}\,,
\end{gather*}
where $\texttt{w}$ and $\texttt{w}'$ are the weight function and its dual, as defined in \eqref{wf} and \eqref{dwf}, respectively.\par
\begin{theorem}\label{dec_helmotz_teorema}{\sl 
Let $q \in (1,+\infty)$. Then, there holds
\begin{equation}
L^q_{\texttt{w}}(\R_+^n)=J^q_{\texttt{w}}(\mathbb{R}_+^n)\oplus G^q_{\texttt{w}}(\mathbb{R}_+^n)\,,
\end{equation}
 that is for all $u \in L^q_{\texttt{w}}(\R_+^n)$, $u=v+\nabla \pi_u$, with the following integral identities:
\begin{gather}
(u,\nabla \pi)=(\nabla \pi_u ,\nabla \pi) \quad \text{for all}\, \, \nabla \pi \in G_{\texttt{w}'}^{q'}(\R_+^n)\, ,  \label{1)} \\
(v,\nabla \pi)=0 \quad \text{for all}\, \, \nabla \pi \in G_{\texttt{w}'}^{q'}(\R_+^n) \label{2)}\, ,
\end{gather}
and
\begin{equation} \label{3}
\dm v\dm _{L^q_\texttt{w}(\R^n_+)} + \dm  \nabla \pi \dm _{L^q_\texttt{w}(\R^n_+)} \leq C \dm u\dm _{L^q_\texttt{w}(\R^n_+)}\,
\end{equation}
with $C$ independent of $u$.}
\end{theorem}
Setting
\[\mathscr{C}_0(\R^n_+):=\{\phi \in \mathit{C}_0^{\infty}(\R^n_+)\, :\, \n \cdot \phi=0\}\,,\]
it's no difficult to prove that $J^{p}_{\texttt{w}}(\R^n_+)$ coincide with the completion of $\mathscr{C}_0(\R^n_+)$ in the $L^p_{\texttt{w}}(\R^n_+)$-norm\footnote{For more details, see \cite{DFP} and references therein. 
}.\par
In \cite{yamazaki}, the author showed that the Helmholtz decomposition also holds for Lorentz spaces. In particular, by extending the projection operator $P$ to the space $L^{p,q}(\R^n_+)$, for $p\in(1,+\infty)$, $q\in [1,+\infty)$, he proved that the direct sum decomposition $L^{p,q}(\R^n_+)=J^{p,q}(\R^n_+)\oplus G^{p,q}(\R^n_+)$ holds. Moreover, for $p$ and $q$ as above, the space $\mathscr{C}_0(\R^n_+)$ is dense in $J^{p,q}(\R^n_+)$ and estimates \eqref{1)}, \eqref{2)} and \eqref{3} still hold with natural adaptations.\par
For the sake of completeness, we recall the following lemma contained, \textit{e.g.}, in \cite{Mlectures}.
\begin{lemma}\label{HDLS}{\sl Let $w\in J^q(\R^n_+)\cap L^{r,\infty}(\R^n_+)$, for some $q,r\in(1,\infty)$. Assume that \be\label{HDLS-I} |(w,\varphi)|\leq A\dm \varphi\dm_{(r',1)}\,,\mbox{ for all }\varphi\in \mathscr C_0(\R^n_+)\,.\ee
Then, there exists a constant $c>0$ such that \be\label{HDLS-II}|(w,\psi)|\leq cA\dm\psi\dm_{(r',1)}\,,\mbox{ for all }\psi\in L^{r',1}(\R^n_+)\,.\ee
In particular, we get $$\dm w\dm_{(r,\infty)}\leq c\,A\,.$$}\end{lemma}
\bp
We claim that the estimate \rf{HDLS-I} initially defined on $\mathscr C_0(\R^n_+)$ can be extended to all $\varphi\in J^{r',1}(\R^n_+)$.
Actually, if $\{\varphi_m\} \subset \mathscr C_0(\R^n_+)$ is convergent to $\varphi$ in $J^{r',1}$, then  we easily get $$\ba{ll}|(w,\varphi)|\hskip-0.2cm&\leq A\dm \varphi\dm_{(r',1)}+A\dm \varphi_m-\varphi\dm_{(r',1)}+|(w,\varphi_m-\varphi)|\VSE \leq A\dm \varphi\dm_{(r',1)}+A\dm \varphi_m-\varphi\dm_{(r',1)}+\dm w\dm_{(r,\infty)}\dm \varphi_m-\varphi\dm_{(r',1)}\,.\ea$$ Letting $m\to\infty$, we realize the goal.
Now, we look for \rf{HDLS-I} with $\psi\in L^{(r',1)}(\R^n_+)$.\par
Assume $\psi\in L^{(r',1)}(\R^n_+)$  as limit of $\{\psi_m\}\subset C_0^1(\R^n_+)$ in $L^{r',1}(\R^n_+)$-norm. For all $m\in\N$, we consider the Helmholtz decomposition of $\psi_m$ in $J^{r',1}$. The decomposition  detects two sequences $\{\varphi_m\}\subset J^{r',1}\cap  J^{q'}(\R^n_+)$ and $\{\n h_m\}\subset L^{r',1}(\R^n_+)\cap L^{q'}(\R^n_+)$, the former, letting $m\to\infty$, furnishes $\varphi\in J^{r',1}(\R^n_+)$ of the Helmholtz decomposition of $\psi$. Both the sequences are in $L^{q'}(\R^n_+)$ as consequence of the compact support of $\psi_m$.\par 
By virtue of the orthogonality between  $w$ and $\n h_m$, recalling that $\varphi_m=\psi_m-\n h_m$ via the Helmholtz decomposition, we get
$$|(w,\psi)|=|(w,\psi-\psi_m)+(w,\psi_m)|=|(w,\psi_m-\n h_m)|=|(w,(\psi-\psi_m)+(w,\varphi_m)|\,.$$  
Applying H\"older's inequality, via \rf{HDLS-I} we arrive at
$$ \ba{ll}|(w,\psi)|\hskip-0.2cm&\leq|(w,\psi-\psi_m)|+|(w,-\varphi+\varphi_m)|+|(w,\varphi)|\VSE\leq \dm w\dm_{(r,\infty)}\Big[\dm \psi-\psi_m\dm_{{(r',1)}}+\dm \varphi-\varphi_m\dm_{{(r',1)}}\Big]+|(w,\varphi)|\VSE \leq \dm w\dm_{(r,\infty)}\Big[\dm \psi-\psi_m\dm_{{(r',1)}}+\dm \varphi-\varphi_m\dm_{{(r',1)}}\Big]+A\dm \varphi\dm_{{(r',1)}}\VSE\leq \dm w\dm_{(r,\infty)}\Big[\dm \psi-\psi_m\dm_{{(r',1)}}+\dm \varphi-\varphi_m\dm_{{(r',1)}}\Big]+cA\dm \psi\dm_{{(r',1)}}\,,\ea$$ where in the last estimate we also employ the estimate $\dm \varphi\dm_{J^{(r',1)}}\leq c\dm\psi\dm_{(r',1)}$.\par
Letting $m\to\infty$ in the last relation, we complete the proof. 
\ep

In the following, we denote by $\mathcal{E}(x) $ the fundamental solution of the Laplace equation, i.e. for $n\geq 3$
\begin{equation}\label{Erapp}
    \mathcal{E}(x) = \frac{1}{n(n-2)V(n) |x|^{2-n}} 
\end{equation}
where $V(n)$ denotes the volume of the unit sphere in $\R^n$, and we denote by $H(t,x)$ the fundamental solution of the heat equation
\begin{equation}\label{Hrapp}
    H(t,x) = \begin{cases}
\frac{1}{(4\pi t)^{\frac{n}{2}}}e^{-\frac{|x|^2}{4t}} & \text{if } \; t>0\\
0 & \text{if } \; t=0.
\end{cases}
\end{equation}
We recall an interpolation inequality for our specific weighted space
\begin{lemma} \label{interpol}
\sl
Let $g \in L^{\infty}(\R^n) \cap L^p_{\texttt{w}}(\R^n)$.There exists a positive constant $c$, independent of $g$, such that
\begin{equation}\label{stima_Lq}
\dm g\dm _q \leq c \dm g\dm _{\texttt{w},p}^{1-\gamma}\dm g\dm _{\infty}^{\gamma }\,, \qquad \mbox{with }\gamma=1-\frac{n}{q}\,,
\end{equation}
for all $q>n$.
\end{lemma}
\begin{proof}
    For the proof, see \cite{DFP}.
\end{proof}
In what follows, we will also need to recall a Gronwall-type inequality established in \cite{MP}. For the sake of brevity, we state only the result relevant to our purposes. A more extensive treatment of the topic can be found in the quoted work and in the references cited therein.\par
Let
\be\label{GRLNI}\mu\in [0,1)\,,\; h(t)\leq \psi(t)+ \intll0t\sfrac{g(\tau,h(\tau))}{(t-\tau)^\mu}d\tau\,,\;h(t)\geq0\,,\;\mbox{ a.e. in }t\in(0,T)\,.\ee
where $\psi(t) \in L(\ssfrac{1}{\mu},\infty)(0,T)$ and $g$ is a function such that, a.e. in $\tau\in(0,T)$, \begin{equation}\label{PgGI}
g(\tau,h(\tau))\leq  A \,h(\tau)\hskip0.05cm |\tau-t_0|^{-\nu}
\end{equation}
with $A\geq 0$, independent of $\tau$, $t_0 \in [0,T)$ and for  suitable exponents $\nu\in[0,1)$.\par
We have the following
\begin{lemma}\label{L-GWSI}{\sl 
Suppose that $\mu+\nu<1$ and   $|t-t_0|^{-\nu}h(t)\in L^1(0,T)$. Then:
\be\label{I-G-I}\Big[\intll0t h^s(\tau)d\tau\Big]^\frac1s\leq t^{\frac1s-\mu}\dm h\dm_{(\frac{1}{\mu},\infty)}\leq C \dm \psi \dm_{(\frac{1}{\mu},\infty)}\,\;\mbox{for all }s\in[1,\ssfrac1\mu)\mbox{ and }t\in [0,T)\,,   \ee 
where $C:= C(A,T)$ is a constant.
}
\begin{proof}
    For the proof of this lemma, see Lemma$\,$3 in \cite{MP}.
\end{proof}
\end{lemma}
\section{Stokes Cauchy problem in \texorpdfstring{$L^p_\texttt{w}(\R^n)$}{Lp-weighted(Rn)}-setting}\label{sect. 3}
We first recall the existence and uniqueness theorem for the Stokes Cauchy problem, together with an auxiliary result suited to our approach.\par
We consider  the Stokes Cauchy problem
\begin{equation}\label{Cproblem}
\begin{array}{ll}
u_t - \Delta u = - \nabla \pi  \,,\;&\text{in}\,\, (0,\infty) \times \mathbb{R}^n\,, \\
\nabla \cdot u=0\,, &\text{in}\,\, (0,\infty) \times \mathbb{R}^n, \\
u(0,x)=u_0(x)\,, 
\end{array}
\end{equation}
under the assumptions that $u_0 \in L^p_{\texttt{w}}(\mathbb{R}^n)$ and $\nabla \cdot u_0=0$ in the weak sense.\par
In \cite{MP}, the authors consider this problem with initial data in $L^p_{\alpha}(\R^n)=\{u\, :\, u|x|^{\alpha} \in L^p(\R^n)\}$ and, among other things, establish results on the existence and uniqueness of a regular solution. In \cite{DFP}, for the Stokes problem, we extend the results considering the more general weight function \eqref{wf}.
\begin{theorem}\label{SPWS}
\sl{
Let consider $p>n\geq 3$. For all $u_0 \in J_{\texttt{w}}^p(\R^n)$ there exists a unique  smooth solution $u(t,x)$ to the Cauchy problem \eqref{Cproblem} such that the following representation holds
\be\label{RHEOu} 
u(t,x)=\int_{\R^n} H(t,x-y)u_0(y)\,dy\,,
\ee
and
\begin{equation}\label{DCWS}
\dm u(t)\dm _{L^{p}_\texttt{w}(\R^n)}\leq \dm u(s)\dm _{L^{p}_\texttt{w}(\R^n)}\,,\mbox{ for all }t>s \geq 0\,.
\end{equation}
Moreover, there exists a constant $c$ independent of $u_0$ and of $s\geq0$, such that for all $l,k \in \N_0$ and for $q \in (n,\infty]$ the following estimate holds
\begin{equation}\label{gradienti_eq_stokes_beta_q}
\dm D_t^l\nabla^k u(t)\dm _{L^q(\R^n)}\leq c(t-s)^{-\frac{k}{2}-l-\frac{n}{2}({\frac{1}{n}}-\frac{1}{q})}\dm u(s)\dm _{L^{p}_\texttt{w}(\R^n)}\,, \text{ for all } t>0\,.
\end{equation} 
Furthermore, the following limit property holds
\begin{equation}\label{convergenza_parte_lineare}
\lim_{t \to 0^+} \dm u(t) - u_0\dm _{L^{p}_\texttt{w}(\R^n)} = 0\,.
\end{equation} 
  Finally,  for all $T>0$ and $q\in (n, +\infty]$, we get \be\label{CRHE}u\in C([0,T);J^p_{\texttt{w}}(\R^n))\cap C((\vep,T); L^q(\R^n)),\mbox{ for all }\vep>0\,.\ee
  }
\end{theorem}
The following result follows the same argument as Lemma$\,8$ and $9$ in \cite{MP}. We give the analogous of these lemmas in our framework. The proof proceeds along the same lines, once the results established in \cite{DFP} are taken into account. Therefore, for the sake of brevity, we omit the details of the proof.
\begin{lemma}
\sl{  Let $g \in L^p_\texttt{w}(\R^n)$. Then, for the convolution product $H\ast g$ the following estimates hold 
  \begin{equation}\label{SLinfSCP}
      t^{\frac{1}{2}}\dm H\ast g(t)\dm _{\infty}\leq K_g(t,\rho),\quad \text{for all}\, \, t>0\mbox{ and }\rho>0\,.
  \end{equation}
In particular, for all $(t,x)\in(0,+\infty)\times\R^n$,
\begin{equation} \label{SPSCP}
     t^{\frac{n}{2p}}|x|^{\alpha}|H\ast g(t,x)|\leq K_g(t,\rho)\,,
\end{equation}
where $K_g(t, \rho)$ is defined in \eqref{krho}.
}
\end{lemma}
\section{IBVP of the Stokes system in the half-space}\label{sect. 4}
Let us consider the following Stokes initial boundary value problem:
\begin{equation}\label{Stokesf}
\begin{array}{ll}
v_t - \Delta v = -  \nabla \pi +  f \,,\;&\text{in}\,\, (0,\infty) \times \mathbb{R}^n_{+}\,, \\
\nabla \cdot v=0\,, &\text{in}\,\, (0,\infty) \times \mathbb{R}^n_{+}, \\
v_{| x_n=0}\,=0 \,,\\
v(0,x)=v_0(x)\,,
\end{array}
\end{equation}
Assuming suitable hypothesis on the datum $f$, it can be represented via the Helmholtz decomposition as
\[
f = Pf + \nabla \Psi\,\,\,\,\,\,\,\, \mbox{with}\qquad
\Psi(x,t) = - \int_{\mathbb{R}^n_+} \nabla_y \big[ \mathcal{E}(x-y) + \mathcal{E}(x-y^*) \big] \cdot f(t,y)\,dy,
\]
where $\mathcal{E}(z)$ is the fundamental solution of the Laplace equation defined in \eqref{Erapp}. Moreover,
\[
\nabla \cdot Pf = 0, \qquad Pf \cdot n\big|_{x_n=0} = 0.
\]
In \cite{CMon, solo, soln}, it has been shown that a solution of problem \eqref{Stokesf} is formally given by
\begin{equation}\label{rap}
v(x,t) =
\int_{\mathbb{R}^n_+} \mathcal{G}(t,x,y)\,v_0(y)\,dy
+
\int_0^t \int_{\mathbb{R}^n_+}
\mathcal{G}(t-\tau,x,y)\,Pf(\tau,y)\,dy\,d\tau,
\end{equation}

\[
\pi(x,t) =
\int_{\mathbb{R}^n_+} \mathcal{Q}(t,x,y)\cdot u_0(y)\,dy
+
\int_0^t \int_{\mathbb{R}^n_+}
\mathcal{Q}(t-\tau,x,y)\cdot Pf(\tau,y)\,dy\,d\tau
+ \Psi(t,x),
\]
where $\mathcal{G} = (G_{ij})_{i,j=1,\dots,n}$ and $\mathcal{Q} = (Q_j)_{j=1,\dots,n}$ are the following functions
\begin{equation}
\begin{aligned}\label{Gkernel}
G_{ij}(t,x,y)
&= \delta_{ij}\big(H(t,x-y) - H(t,x-y^*)\big)
+ 4\,\delta_{jn}\,\partial_{x_i}
\int_{\mathbb{R}^n_+}
\partial_{z_n}\mathcal{E}(x-z)\,H(t,x-z^*)\,dz \\
&= \delta_{ij}\big(H(t,x-y) - H(t,x-y^*)\big)
+ G^{*}_{ij}(t,x,y)\,,
\end{aligned}
\end{equation}
\be Q_j(t,x,y)=4\hat{\delta}_{rj}D_{x_r} \left[ \int_{\R^{n-1}}\!\!\!\!\!\! D_{x_n} \mathcal{E}(x-z')H(t,z'-y',y_n) dz' + \int_{\R^{n-1}}\!\!\!\!\!\! \mathcal{E}(x-z')D_{y_n}H(t,z'-y', y_n) dz'\right],\label{Q}\ee
where $\hat{\delta}_{rj} =1- \delta_{rj}\,$.\par
The functions $G^*_{ij}$ and $Q_j$ satisfy the following pointwise estimate
\begin{equation}\label{SPNI}
\begin{array}{cc}
\vspace{5pt}
     |D^{h,l,k}_{x,y,t} G^*_{ij}(t,x,y)|\leq c\, t^{-\frac{l_n}{2}-k}(x_n^2 + t)^{\frac{h_n}{2}} (|x-y^*|^2 +t)^{-\frac{n+|h'|+|l'|}{2}}e^{-\hat{a}\frac{y_n^2}{t}}\,, \\ 
    |D^{h,l,k}_{x,y,t} Q_{j}(t,x,y)|\leq c\, t^{1-\frac{l_n}{2}-k} (|x-y^*|^2 +t)^{-\frac{n-1+|h|+|l'|}{2}}e^{-\hat{a}\frac{y_n^2}{t}}\,,
\end{array}
\end{equation}
for any $|h|,|l|, k \in \N \cup\{0\}$ with $h=(h',h_n)$, $l=(l',l_n)$, and $\hat{a}>0$. Therefore, taking into account the following property of the heat kernel
\begin{equation}
       |D^{h,l,k}_{x,y,t} H(t,x-y)|\leq c\, t^{-\frac n2-\frac{|l|}{2}-\frac {|h|}{2}} e^{a\frac{|x-y|^2}{t}}\leq c \,(|x-y|^2 +t)^{-\frac{n+|h|+|l|+2k}{2}}\,,
\end{equation}
we have that
\begin{equation}\label{SKIC}
       |D^{h,l,k}_{x,y,t} G_{ij}(t,x,y)|\leq c \,(|x-y|^2 +t)^{-\frac{n+|h|+|l|+2k}{2}}+ c\, t^{-\frac{l_n}{2}-k}(x_n^2 + t)^{\frac{h_n}{2}} (|x-y^*|^2 +t)^{-\frac{n+|h'|+|l'|}{2}}e^{-\hat{a}\frac{y_n^2}{t}}\,
\end{equation}
since $H(t, x-y)$ satisfies \eqref{SPNI} as well.\par
\begin{lemma}\label{lemmaStimeyoungKernell}
\sl{
    Let $\mathcal{G} = (G_{ij})_{i,j=1,\dots,n}$ be the kernel appearing in \eqref{rap}, where $G_{i,j}$ is defined in \eqref{Gkernel}. Then, for $r\geq1$, the following estimate holds  \begin{equation}\label{StimeyoungKernellx}
      \sup_{x\in \R^n_+}\dm D^{h,l,k}_{x,y,t} \, G_{i,j}(t,x)\dm_{r}\leq c\, t^{-\frac{n+|h|+|l|+2k}{2}+\frac{n}{2r}}\,,
    \end{equation}
and, similarly,
\begin{equation}\label{StimeyoungKernelly}
      \sup_{y\in \R^n_+}\dm D^{h,l,k}_{x,y,t} \, G_{i,j}(t,y)\dm_{r}\leq c\, t^{-\frac{n+|h|+|l|+2k}{2}+\frac{n}{2r}}\,,
    \end{equation}
for each fixed $t>0$.}
\end{lemma}
\begin{proof}
    We only prove the estimate in the case $|h|=1$ and $|l|=k=0$, since the general case can be treated in the same way.\par
    Considering the zero extension of $\n\, G$ to $\R^n$ and taking into account the pointwise estimate \eqref{SKIC}, we have
\begin{equation}\label{b}
\begin{split}
    \dm \nabla G(t,x)\dm _{r}^{r} & \leq \int_{\R^n} \frac{1}{\left(|x-y|^2 + t \right)^{\frac{n+1}{2}r}} \, dy + \int_{\R^n} \frac{e^{-\frac{\hat{a}y^2_n}{t}r}}{\left(|x-y^*|^2 + t \right)^{\frac{n}{2}r} t^{\frac{r}{2}}} \, dy =: I_1 + I_2. 
\end{split}
\end{equation}
About the first term, by simple computations, we have
\begin{equation}
  I_1  = \frac{1}{t^{\frac{n+1}{2}r}}\int_{\R^n} \frac{1}{\left(\frac{|x-y|^2}{t} + 1 \right)^{\frac{n+1}{2}r}} \, dy =\frac{1}{t^{\,\frac{(n+1)r-n}{2}}}\int_{\R^n} \frac{1}{\left(|z|^2 + 1 \right)^{\frac{n+1}{2}r
  }} \, dy \leq \frac{c}{t^{\,\frac{(n+1)r-n}{2}}}\,, \label{I_1}
\end{equation}
with $c$ independent of $x$. \par
Concerning the second term, recalling that $x':=(x_1, ..., x_{n-1}),$ we have 
\begin{align*}
    I_2&= \frac{1}{t^{\frac{r}{2}}} \int_{-\infty}^{+\infty} e^{-\hat{a}\frac{y^2_n}{t}r} \int_{\R^{n-1}} \frac{1}{\left(|x'-y'|^2 + x_n^2 + y_n^2+ t \right)^{\frac{n}{2}r} } \, dy=\frac{1}{t^{\frac{r}{2}}} \int_{-\infty}^{+\infty} \frac{ e^{-\hat{a}\frac{y^2_n}{t}q'} }{ (x_n^2 + y_n^2+ t )^{\frac{n(r-1)+1}{2}}} dy_n,
\end{align*}
where the last integral has been computed by the change of variables $z'= \frac{x'-y'}{t^{\frac{1}{2}}}$.\par
Now,
\begin{equation}
    \label{I_2}
    I_2  \leq \frac{1}{t^{\frac{r}{2}-\frac{1}{2}+\frac{n(r-1)+1}{2}} }
    \int_{-\infty}^{+\infty} \frac{e^{-\hat{a}\frac{y^2_n}{t}r}}{t^\frac{1}{2}} dy_n \leq\frac{c}{t^{\,\frac{(n+1)r}{2}-\frac{n}{2}}}.
\end{equation} 
with $c$ independent of $x$.\par
So, using \eqref{I_1} and \eqref{I_2} in \eqref{b}, we have
\begin{equation}\label{c}
    \dm \nabla G(t,x)\dm_{r} \leq \frac{c}{t^{\,\frac{n+1}{2}-\frac{n}{2r}}}\,,
\end{equation}
from which the claim follows. Equivalently, one obtains \eqref{StimeyoungKernelly}.
\end{proof}
In \cite{Solonn}, it has been shown that if there exists a tensor $F$ such that
\[
f_j = \nabla \cdot F_j
\quad \text{and} \quad
F_{nj}\big|_{x_n=0}=0, 
\qquad j=1,\dots,n,
\]
where by $F_j$ we denote the $j$th-column of $F$, then $Pf$ can be written as
\[
\mathtt{f}_j = (Pf)_j = \nabla \cdot \mathcal{F}_j,
\]
with
\begin{align*}
\mathcal{F}_{nj}&(x,t) = F_{nj}(x,t) - \delta_{nj}F_{nn}(x,t),
\qquad j=1,\dots,n,\\
\VS
\mathcal{F}_{ms}&(x,t)
=
F_{ms}(x,t) - \delta_{ms}F_{nn}(x,t) \\
  & + D_{x_s}
\left[
\int_{\mathbb{R}^n_+}
\Big(
D_{y_q}\mathcal{N}(x,y)F_{mq}(y,t)
+
D_{y_n}\mathcal{N}(x,y)F_{nm}(y,t)
-
D_{y_m}\mathcal{N}(x,y)F_{nn}(y,t)
\Big)\,dy
\right],\\
\VS
\mathcal{F}_{mn}&(x,t)
=
-\hat{\delta}_{rs}
\int_{\mathbb{R}^n_+}
D_{x_n}\mathcal{N}(x,y)F_{ms}(y,t)\,dy
+
D_{x_m}
\int_{\mathbb{R}^n_+}
D_{x_n}\mathcal{N}(x,y)F_{nn}(y,t)\,dy \\
& - F_{nm}(x,t)
- \hat{\delta}_{rs}D_{x_r}
\int_{\mathbb{R}^n_+}
D_{y_s}\mathcal{N}^{-}(x,y)
\big(F_{mn}(y,t)+F_{nm}(y,t)\big)\,dy,
\end{align*}
where
\be \label{N}
\mathcal{N}(x,y)=\mathcal{N}^{+}(x,y)
= \mathcal{E}(x-y)+\mathcal{E}(x-y^{*}),
\qquad
\mathcal{N}^{-}(x,y)
=\mathcal{E}(x-y)-\mathcal{E}(x-y^{*}).
\ee

Thus,
\begin{equation*}
\mathtt{f}_j
=
D_{x_i}F_{ij}
-
D_{x_j}F_{nn}
-
\hat{\delta}_{rm}\hat{\delta}_{nj}
D_{x_r}(F_{mn}+F_{nm})
+
h_j,
\end{equation*}

where $h_j$ is a linear combination of terms of the kind
\be \label{hj}
D_{x_m x_s}
\int_{\mathbb{R}^n_+}
D_{y_q}\mathcal{N}^{\pm}(x,y)F_{pr}(y,t)\,dy,
\qquad m,s \neq n.
\ee

Therefore, if we also assume that $F_{ij \,| x_n=0}=0$, then after integration by parts, we can write the last term on the right-hand side of $\eqref{rap}_1$ in the form
\be \label{rapfab}
\begin{split}
&\int_0^t \int_{\mathbb{R}^n_+}
G_{ij}(t-\tau,x,y)\mathtt{f}_j(y,\tau)\,dy\,d\tau
=\\
&-\int_0^t \int_{\mathbb{R}^n_+}
D_{y_p}G_{ij}(t-\tau,x,y)F_{pj}(y,\tau)\,dy\,d\tau 
+
\int_0^t \int_{\mathbb{R}^n_+}
D_{y_j}G_{ij}(t-\tau,x,y)F_{nn}(y,\tau)\,dy\,d\tau\\
&+
\hat{\delta}_{rm}
\int_0^t \int_{\mathbb{R}^n_+}
D_{y_r}G_{in}(t-\tau,x,y)
(F_{mn}+F_{nm})(y,\tau)\,dy\,d\tau+
\int_0^t \int_{\mathbb{R}^n_+}
G_{ij}(t-\tau,x,y)h_j(y,\tau)\,dy\,d\tau. 
\end{split}
\ee
Now, integrating by parts twice and taking in account that $D_{x'}G(s,x,y)=- D_{y'}G(s,x,y)$, the last term can be written as a linear combination of the expressions
\small{\begin{equation}\label{lincomK}
D_{x_s x_m} \!\!\int_0^t \!\!\!\int_{\R^n_+}\!\!\!\! G_{ij}(t \!-\!\tau,x,y)\!\bigg(\!\int_{\R^n_+}\!\!\!\!\!\!D_{y_q}\mathcal{N}^{\pm}(y,z)F_{pr}(z,\tau)\, dz \!\bigg) dy \,d\tau \!= \!\!\int_0^t \!\!\!\int_{R^n_+}\!\!\!\!\!\! D_{x_s x_m} K_{ijq} (t\! -\!\tau, x,z) F_{pr}(\tau, z) \, dz\,,
\end{equation}}
where
\begin{equation}\label{Kdef}
    K_{ijq}(s,x,z)=\int_{\R^n_+}G_{ij}(s,x,y) D_{y_q}\mathcal{N}^{\pm}(y,z)\,dy\,.
\end{equation}
In order to derive suitable estimates for the solution to the problem, we recall below some useful properties of this kernel.
\begin{lemma}\label{lemma9SPK}
    \sl{Let $l,h\in \N^n$ be multi-indices with $|l|,|h|\geq0$. Write $h=(h',h_n)$ and $l=(l',l_n)$ , where $l',h'\in \N^{n-1}$. Then, the kernel $K_{ijq}$ satisfies the estimate
    \begin{equation}\label{SKKCM}
       |D_{s,x, z' }^{k,h,l'}K_{ijq}(s,x,z)|\leq \ssfrac{c}{(|x-z|^{2}+s)^{\frac{n-1+|h|+|l'|+2k}{2}}}\,+ cs^{-k}\ssfrac{(x_n^2+s)^{-\frac{h_n}{2}}}{(|x-z|^{2}+s)^{\frac{n-1+|h'|+|l'|}{2}}}+c\ssfrac{[(x_n-z_n)^2+s]^{-\frac{h_n+2k}{2}}}{(|x-z|^{2}+s)^{\frac{n-1+|h'|+|l'|}{2}}}\,.    
    \end{equation}
    uniformly in $x$ and $z\in \R^n_+$, where  $s>0$, provided that $|l'|=0,1$. In particular, for $|l'|=k=0$, we have
    \begin{equation}\label{SKKS}
       |D_{x'}^{h'}K_{ijq}(s,x,z)|\leq \frac{c}{(|x-z|^{2}+s)^{\frac{n-1+|h'|}{2}}}\,. 
    \end{equation}
    }
\end{lemma}
\begin{proof}
    We refer to \cite{CMon} for the proof of estimate \eqref{SKKCM}. The second one is due to Solonnikov and was proved in \cite{Solonn}, Proposition$\,$3.1.
\end{proof}

\begin{lemma}\label{stimalrKlemma}
\sl{
Let $K$ be the kernel defined in \eqref{Kdef}, and let $l$ and $k$ be multi-indices as in the previous lemma. Then, for $r\geq 1$, the following estimate holds
\begin{equation}\label{stimeLrKx}
    \sup_{x\in \R^n_+}\dm D_{s,x, z' }^{k,h,l'}K(s,x)\dm_r \leq c\,s^{-\frac{n+|h|+|l|+|l'|-1}{2}+\frac{n}{2r}}\,,
\end{equation}
}
and, similarly,
\begin{equation}\label{stimeLrKz}
    \sup_{z\in \R^n_+}\dm D_{s,x, z' }^{k,h,l'}K(s,z)\dm_r \leq c\,s^{-\frac{n+|h|+|l|+|l'|-1}{2}+\frac{n}{2r}}\,,
\end{equation}
for all fixed $s>0$.
\end{lemma}
\begin{proof}
    The pointwise estimate \eqref{SKKCM} allows one to repeat the proof of the previous lemma; hence, the details are omitted.  
\end{proof}
\subsection{The case \texorpdfstring{$f=0$}{f=0}}
The previous paper \cite{DFP} is devoted to the study of IBVP for the Stokes system in the half-space in the same Lebesgue weighted functional spaces. In this section, we summarize the results obtained and refine them to suit the purposes of the present work.\par
We start to recall a classical result regarding the problem \eqref{Stokesf} with $f\equiv 0$.
\begin{lemma}\label{SHE}{\sl Let $\varphi_0\in \mathscr C_0(\R^n_+)$. Then we get a unique solution to problem \rf{Stokesf}, with $f\equiv 0$, such that
\be\label{RHE}\varphi(t, x)\in C^k(0,T;{\underset{q>1}\cap} J^q(\R^n_+))\cap({\underset{q>1}\cap} L^q(0,T;W^{2,q}(\R^n_+))\,,\mbox{ for }k\in\N_0, \mbox{ and for all }T>0\,.\ee
Moreover, the following representation formula holds:\be\label{RHEO} 
\varphi(t,x)=\int_{\R^n_+} \mathcal{G}(t,x,y)\varphi_0(y)\,dy\,.
\ee
In particular, for all $r\geq q>1$, there exists a constant $c$ such that \be\label{SP}\dm \varphi(t)\dm_r+(t-s)^\frac12\dm\n \varphi(t)\dm_r\leq c(t-s)^{-\frac n2\left(\frac1q-\frac1r\right)}\dm \varphi(s)\dm_q\,,\mbox{ for all }t>s\geq0\,.\ee Finally,  for  $q\in(1,\infty)$ and $r\in[1,\infty)$, we get 
\be\label{SPLS}\ba{c}\varphi\in C([0,T);L^{q,r}(\R^n_+))\,,\VS\dm \varphi(t)\dm_{(q,r)}+(t-s)^\frac12\dm\n \varphi(t)\dm_{(q,r)}\leq c\dm \varphi(s)\dm_{(q,r)}\,\mbox{ for all }t>s\geq0\,.\ea\ee
}\end{lemma}\bp The existence, the  uniqueness and properties \rf{RHE}-\rf{SP} are classical results and we refer to monograph \cite{Solonn}. Property \rf{SPLS}    follows from \rf{RHE} and \eqref{SP} via real interpolation.\ep
In \cite{DFP}, we considered the problem \eqref{Stokesf} with $f\equiv0$, under the assumptions that $v_0 \in L^p_{\texttt{w}}(\mathbb{R}^n_{+})$, and $\nabla \cdot v_0=0$ in weak sense. The main results obtained in the quoted work, useful in the present investigation, are the following:
\begin{theorem}\label{Theoremprincipale}
\sl
Let  consider $p>n\geq 3$. Let $w$ be the weight function defined in \eqref{wf}. Then, for all $v_0\in J^p_{\texttt{w}}(\R^n_+)$ there exists a unique smooth solution to the problem \eqref{problem} and admits the represention
\begin{equation}
    v(t,x)=\int_{\R_n^+ }\mathcal{G}(t,x,y)v_0(y)\,dy\,.
\end{equation}
Moreover, there exists a constant $c$ independent of $v_0$ and of $s\geq 0$, such that
\be \label{smgP}
\dm v(t)\dm_{\texttt{w},p}\leq c \,\dm v(s)\dm_{\texttt{w},p}, \quad \mbox{for all } t>s\geq 0, 
\ee
for any $l, k \in \mathbb{N}_0$, $q\in (n,+\infty]$ and for all $t>0$, the following estimate holds
\be \label{RP}
\dm D_t^l \nabla^k v(t)\dm_q
\leq c\, (t-s)^{-\frac{n}{2}(\frac{1}{n}-\frac{1}{q})-\frac{k}{2}-l}\dm v(s)\dm_{\texttt{w},p},  \quad \mbox{for } t>s\geq 0\,,
\ee
and the pressure term fulfills the following estimate
\begin{equation}\label{pressure}
 \dm  \nabla \pi_v(t)\dm_q\leq c\, (t-s)^{-\frac{n}{2}(\frac{1}{n}-\frac{1}{q})-1}\dm v(s)\dm_{\texttt{w},p}\,, \quad \mbox{for } t>s\geq 0\,.  
\end{equation}
Additionally, the solution satisfies the following initial condition
\be \label{LP}
\lim_{t \to 0^+} \dm v(t)-v_0 \dm_{L^p_{\texttt{w}}(\R^n_+)}=0.
\ee
Finally, for any $q\in (n,\infty]$, $\varepsilon>0$, and $T>0$ we have
\be \label{regular_Hs}
v\in C([0,T);J^p_{\texttt{w}}(\R^n_+))\cap C((\varepsilon,T); L^q(\R^n_+))\,.
\ee
\end{theorem}

\begin{coro}\label{andamenti_asintotici}
\sl{
Let $v(t,x)$ be the solution given in the Theorem \ref{Theoremprincipale} and let $\alpha=1-\frac{n}{p}$, with $p>n\geq3$. Then, there exists a suitable $R>0$ such that 
\begin{equation}\label{pe}
    |v(t,x)|\leq c  \, |x|^{-\alpha}\, t^{-\frac{n}{2p}}\,\dm v_0\dm_{L^p_{\texttt{w}}(\R^n_+)}\,, \quad \mbox{for all }t>0\,,\mbox{ and }|x|>4R\,. 
\end{equation}}
\end{coro}
We now proceed to the statement and proof of the following results: the former is equivalent to Lemma$\,3$ in the half-space setting, while the latter will allow us to handle the linear part in the next section.\par
For $g\in L^p_{\texttt{w}}(\R^n_+)$, let us define the following function
\be \label{snellimento}
\vf(t,x)= \int_{\R^n_+} \mathcal{G}(t,x,y)g(y)\,dy\,.
\ee
The following results hold
\begin{lemma}\label{piclinf} \sl{
    Let $K_g(t, \rho)$ be the quantity defined in \eqref{krho} and $g\in L^p_{\texttt{w}}(\R^n_+)$. For $\vp$ defined as in \eqref{snellimento}, the following estimates hold
    \begin{equation}\label{stimapiclinf}
        t^{\frac{1}{2}} \dm \vf(t,x)\dm_{\infty} \leq K_g(t,\rho) \hspace{0.7cm} \text{for all } t>0, \rho>0,
    \end{equation}
and, for all $(t,x)\in(0,+\infty)\times\R^n_+$,
\begin{equation}\label{picpunt}
    t^{\frac{n}{2p}}|x|^{\alpha} |\vp (t,x)| \leq K_g(t,\rho)
\end{equation}
with $\alpha=1-\frac np$.
}
\end{lemma}
\begin{proof}
Proceeding as in \cite{DFP}, we write
\begin{align*}
    \vf(t,x)=& \int_{\R^n_+} [H(t, x-y)- H(t, x-y^*)] g(y) \, dy + \int_{\R^n_+} G^* (t,x,y) g(y)\, dy \\
= & \int_{\R^n_+}[H(t, x-y)- H(t, x-y^*)]g(y) \, dy \\
& +  4 \sum_{i=1}^{n-1} \partial_{x_i} \int_0^{x_n}
\int_{\mathbb{R}^{n-1}}
\nabla \mathcal{E}(x-y) \int_{\R^n_+}\,H(t,y-z^*)g(z)\,dz\, dy \, .
\end{align*}
Denoting by $g^* (y)$ the odd extension of $g$ in the whole space, we can rewrite the first term as follow
$$\int_{\R^n_+}[H(t, x-y)- H(t, x-y^*)]g(y) \, dy = \int_{\R^n}H(t, x-y) g^*(y) \, dy\,.$$
For the second term, we consider $\hat{g}$ the extension of $g$ defined by
$$\hat{g}(x)=\begin{cases}
    0   &\mbox{if }x_n>0\,,\\
    g(x^*) & \mbox{if }x_n \leq 0\,.
\end{cases}$$
Using the well known semigroup property for the solution of the heat equation, we have
\begin{align*}
&\int_{\R^n_+} G^* (t,x,y) g(y)\, dy =  4 \sum_{i=1}^{n-1} \partial_{x_i} \int_0^{x_n}
\int_{\mathbb{R}^{n-1}}
\nabla \mathcal{E}(x-y) \int_{\R^n}\,H(t,y-z)\hat{g}(z)\,dz\, dy\\
 &= 4 \sum_{i=1}^{n-1} \partial_{x_i} \int_0^{x_n}
\int_{\mathbb{R}^{n-1}}
\nabla \mathcal{E}(x-y) \left( \int_{\R^n} H\left(\frac{t}{2}, y,\xi\right) \hat{\psi}\left(\frac{t}{2},\xi\right) \, d\xi \right)  dy = \int_{\R^n_+} G^* \left(\frac{t}{2} x, y\right) \hat{\psi}\left(\frac{t}{2},y\right) \, dy,
\end{align*}
where $\hat{\psi}\left(\frac{t}{2}, y\right)$ is the solution of Stokes system in the whole space with initial datum $\hat{g}(z)$. Using \eqref{SLinfSCP}, we get
\begin{align*}
|\vp(t,x)|\leq & \int_{\R^n}H(t, x-y) |g^*(y) |\, dy +  \int_{\R^n_+} \left| G^* \left( x, y, \frac{t}{2}\right) \right| \left| \hat{\psi}\left(\frac{t}{2},y\right)\right|\, dy\\
& \leq t^{-\frac{1}{2}} \, K_g(t, \rho)+ t^{-\frac{1}{2}}\, K_g(t, \rho) \int_{\R^n_+} \frac{e^{-\frac{\hat{a}y_n^2}{t}}}{\left[ |x-y^*|^2+\frac{t}{2} \right]^{\frac{n}{2}}}\, dy \leq t^{-\frac{1}{2}}\,K_g(t, \rho)\,,
\end{align*}
from which we deduce \eqref{stimapiclinf}.\par 
To obtain \eqref{picpunt}, we use \eqref{SPSCP} and, proceeding as before, we get
\begin{align*}
|\vp(t,x)|\leq & \int_{\R^n}H(t, x-y) |g^*(y) |\, dy +  \int_{\R^n_+} \left| G^* \left( x, y, \frac{t}{2}\right) \right| \left| \hat{\psi}\left(\frac{t}{2},y\right)\right| \frac{t^{\frac{n}{2p}}|y|^\alpha}{t^{\frac{n}{2p}}|y|^\alpha}\, dy\\
& \leq K_g(t, \rho) \int_{\R^n_+} \frac{e^{-\frac{\hat{a}y_n^2}{t}}}{\left[ |x-y^*|^2+\frac{t}{2} \right]^{\frac{n}{2}}} \frac{1}{t^{\frac{n}{2p}}|y|^\alpha}\, dy  \leq t^{-\frac{n}{2p}} |x|^{-\alpha} K_g(t, \rho)\,,
\end{align*}
which completes the proof.
\end{proof}
\begin{lemma}\label{LemmapiccLq}
 \sl{
    Let $q>n$ and let $K_g(t, \rho)$ as in \eqref{krho}. For all $g \in L^p_{\texttt{w}}(\R^n_+)$ and divergence free in a weak sense, $\vp $ defined in \eqref{snellimento} satisfies
    \begin{equation}\label{piccLq}
        t^{\frac n2(\frac 1n - \frac 1q)}\dm\vp(t)\dm_{q}\leq \dm g\dm_{\texttt{w},p}^{1-\gamma}\, K_g(t,\rho)^{\gamma}\,,
    \end{equation}
    and
    \begin{equation}\label{piccgradLq}
        t^{\frac n2(\frac 1n - \frac 1q)+ \frac 12 }\dm \nabla \vp(t)\dm_{q}\leq\dm g\dm_{\texttt{w},p}^{1-\gamma}\, K_g(t,\rho)^{\gamma}\,, 
    \end{equation}
    for all $t>0$, and $\rho>0$, with $\gamma=1-\frac nq$.
    }
\end{lemma}
\begin{proof}
We can consider the solution $v(t,x)$ to problem \eqref{Stokesf} with
$f\equiv0$ and $g$ as the initial datum. By uniqueness established in Theorem$\,$\ref{Theoremprincipale}, $\vp \equiv v$. Then, by estimate \eqref{RP} with $|l|=|k|=0$, estimate \eqref{smgP}, by Lemma$\,$\ref{interpol} and by Lemma$\,$\ref{piclinf}, for all $t>0$, we have
$$t^{\frac n2(\frac 1n - \frac 1q)}\dm \vp(t)\dm_q \leq \dm g\dm_{\texttt{w},p}^{1-\gamma}\, K_g(t,\rho)^{\gamma}$$
which is the estimate \eqref{piccLq}.\par
In order to obtain estimate \eqref{piccgradLq}, it is enough to observe that, for all $t>s\geq 0$, we have
$$ \dm \n \vp(t) \dm_q\leq c\, (t-s)^{-\frac 12 } \dm \vp(s)\dm_q\,.$$
Again, by estimate \eqref{RP}, choosing $s=\frac t2$ and arguing as before, we complete the proof.
\end{proof}

\subsection{The case \texorpdfstring{$f=\n \cdot (a \otimes b)$}{f = n · (a ⊗ b)}}
Let $a$ and $b$ be divergence free vector fields such that\footnote{This restriction on $q$ is due to our  purposes. In general, the theory developed in \cite{KS} and \cite{Solonn} is valid for all $q>1$} $a \otimes b \in L^q((0,T)\times \R^n_+)$, for $q>n$, and assume that  $a\otimes b_{|_{x_n=0}}=0$ .\par In the following, we consider the system \eqref{Stokesf} for $f$ in the form $f= \n \cdot (a \otimes b)$, and $v_0(x)\equiv 0$, i.e.
\begin{equation}\label{Stokesab}
\begin{array}{ll}
\bar{u}_t - \Delta \bar{u}=+  \n \cdot (a \otimes b )- \nabla \pi\,,\;&\text{in}\,\, (0,\infty) \times \mathbb{R}^n_{+}\,, \\
\nabla \cdot \bar{u}=0\,, &\text{in}\,\, (0,\infty) \times \mathbb{R}^n_{+}, \\
\bar{u}_{| x_n=0}\,=0 \,,\\
\bar{u}(0,x)=0\, , & \text{in } \lbrace 0 \rbrace \times \R^n_+.
\end{array}
\end{equation}
In \cite{solo}, Solonnikov proved the existence and uniqueness of a solution to such problem, say $\ov{u}(t,x)$, and furnished the representation formula 
\be \label{ra}
\bar{u}(x,t) = \int_0^t \int_{\R^n_+} \mathcal{G}(t-\tau, x, y) Pf(\tau, y)\, dy d\tau ,
\ee
where, we recall that $Pf$ is the projection of $f$. As recalled above, for $f$ as in our assumptions the formula \eqref{rapfab} becomes
\be\label{rapfabmod}
\begin{split}
&\bar{u}(x,t) =\int_0^t \int_{\mathbb{R}^n_+}
G_{ij}(t-\tau,x,y)[P\n \cdot(a \otimes b)]_j(y,\tau)\,dy\,d\tau\\
&= -\int_0^t \int_{\mathbb{R}^n_+}
D_{y_p}G_{ij}(t-\tau,x,y)a_pb_j(y,\tau)\,dy\,d\tau  +
\int_0^t \int_{\mathbb{R}^n_+}
D_{y_j}G_{ij}(t-\tau,x,y)a_nb_n(y,\tau)\,dy\,d\tau \\
& +
\hat{\delta}_{rm}
\int_0^t \int_{\mathbb{R}^n_+}
D_{y_r}G_{in}(t-\tau,x,y)
(a_mb_n+a_nb_m)(y,\tau)\,dy\,d\tau +
\int_0^t \int_{\mathbb{R}^n_+}
G_{ij}(t-\tau,x,y)h_j(y,\tau)\,dy\,d\tau\,,
\end{split}
\ee
where $h_j$ is defined as linear combination of terms in the form \eqref{hj}.
\begin{lemma}\label{stimaq}
\sl{Let consider $\bar{u}(t,x)$ as in \eqref{rapfabmod} and assume that
$$ \sup_{(0,T)} t^{\frac{n}{2} \left( \frac{1}{n}-\frac{1}{q}\right)} \dm a(t) \dm_{L^q(\R^n_+)} + \sup_{(0,T)} t^{\frac{n}{2} \left( \frac{1}{n}-\frac{1}{q}\right)} \dm b(t) \dm_{L^q(\R^n_+)}< \infty.$$
Then, there exists a constant $c>0$ such that
\begin{equation}\label{stimaqformula}
t^{\frac{n}{2} \left( \frac{1}{n}-\frac{1}{q}\right)} \dm \bar{u}(t) \dm_{L^q(\R^n_+)} \leq c \sup_{(0,T)} t^{\frac{n}{2} \left( \frac{1}{n}-\frac{1}{q}\right)} \dm a(t) \dm_{L^q(\R^n_+)} \dm b(t) \dm_{L^q(\R^n_+)},    
\end{equation}
for all $t \in (0,T).$}
\end{lemma}

\begin{proof}
We limit ourselves to estimate the first term on the right-hand side of \eqref{rapfabmod}. The others follow in the same way.\par
Taking $r=q'$ in estimate \eqref{StimeyoungKernellx}, with $|h|=1$ and $|l|=k=0$  and applying Young's inequality, we obtain
\begin{equation*}\label{a}
\begin{split}
 \dm \ov u (t)\dm _q  
&\leq c  \int_0^t (t-\tau)^{-\frac{1}{2}-\frac{n}{2q}}\dm |a(\tau)||b(\tau)|\dm _{\frac{q
}{2}}\, d\tau \leq  \sup_{(0,t)}\tau^{n(\frac{1}{n}-\frac{1}{q})}\dm |a(\tau)||b(\tau)|\dm _{\frac{q
}{2}}\int_0^t \tau^{-n(\frac{1}{n}+\frac{1}{q})}(t-\tau)^{-\frac{1}{2}-\frac{n}{2q}}\, d\tau\\
&\leq c\, t^{-\frac{n}{2}(\frac{1}{n}-\frac{1}{q})}\sup_{(0,t)}\tau^{n(\frac{1}{n}-\frac{1}{q})}\dm a(\tau)\dm _q\dm b(\tau)\dm _q.
\end{split}
\end{equation*}
which implies the thesis.
\end{proof}
\begin{lemma}\label{LEmmaPNLGRADLq}
\sl{Let consider \eqref{rapfabmod} and assume that
$$ \sup_{(0,T)} \tau^{\frac{n}{2} \left( \frac{1}{n}-\frac{1}{q}\right)} \dm a(\tau) \dm_{q} + \sup_{(0,T)} \tau^{\frac{n}{2} \left( \frac{1}{n}-\frac{1}{q}\right)} \dm b(\tau) \dm_{q}  + \sup_{(0,T)} \tau^{\frac{n}{2} \left( \frac{1}{n}-\frac{1}{q}\right) +\frac{1}{2}} \dm \n b (\tau) \dm_{q}< \infty.$$
Then, there exists a constant $c>0$ such that
\begin{equation}\label{stimaqformulagrad}
\begin{split}
t^{\frac{n}{2} \left( \frac{1}{n}-\frac{1}{q} \right) +\frac{1}{2}} \dm \n \bar{u}(t) \dm_{L^q(\R^n_+)} \leq & \; c \sup_{(0,T)} \tau^{\frac{n}{2} \left( \frac{1}{n}-\frac{1}{q}\right)} \dm a(\tau) \dm_{q}\,\sup_{(0,T)} \tau^{\frac{n}{2} \left( \frac{1}{n}-\frac{1}{q}\right)} \dm b(\tau) \dm_{q} \\
& + c \sup_{(0,T)} \tau^{\frac{n}{2} \left( \frac{1}{n}-\frac{1}{q}\right)} \dm a(\tau) \dm_{q}\, \sup_{(0,T)} \tau^{\frac 12 +  \frac{n}{2} \left( \frac{1}{n}-\frac{1}{q}\right)} \dm \n b(\tau) \dm_{q},
\end{split}
\end{equation}
for all $t \in (0,T).$}
\end{lemma}  
\begin{proof}
    First, as done at beginning of this section, integrating twice and taking in account that $D_{x'}G(s,x,y)=-D_{y'}G(s,x,y)$,  we can rewrite the last term in \eqref{rapfabmod} as a linear combination of terms of the form \eqref{lincomK} with $F_{pr}=a_pb_r$. Denoting by $T(t,x)$ the linear combination thus obtained, we can rewrite \eqref{rapfabmod} as follows
\be \label{rapfabmodmod}
\begin{split}
\bar{u}(x,t) 
&= -\int_0^t \int_{\mathbb{R}^n_+}
D_{y_p}G_{ij}(t-\tau,x,y)a_pb_j(y,\tau)\,dy\,d\tau \,+
\int_0^t \int_{\mathbb{R}^n_+}
D_{y_j}G_{ij}(t-\tau,x,y)a_nb_n(y,\tau)\,dy\,d\tau \\
&  \; \; \; \;+
\hat{\delta}_{rm}
\int_0^t \int_{\mathbb{R}^n_+}
D_{y_r}G_{in}(t-\tau,x,y)
(a_mb_n+a_nb_m)(y,\tau)\,dy\,d\tau +
T(t,x)\,. 
\end{split}
\ee
Differentiating the last formula, we obtain
\be\label{rapfabmodgrad}
\begin{split}
\n_x\ov{u}(t,x)&=
-\int_0^t \!\!\int_{\mathbb{R}^n_+}\!\!\!\!\!
\n_xG_{ij}(t-\tau,x,y)D_{y_p}a_pb_j(y,\tau)\,dy\,d\tau+
\int_0^t \!\!\int_{\mathbb{R}^n_+}
\!\!\!\!\!\n_xG_{ij}(t-\tau,x,y)D_{y_j}a_nb_n(y,\tau)\,dy\,d\tau \\
&+
\hat{\delta}_{rm}
\int_0^t \!\! \int_{\mathbb{R}^n_+}
\!\!\n_xG_{in}(t-\tau,x,y)
D_{y_r}(a_mb_n+a_nb_m)(y,\tau)\,dy\,d\tau +
\n_x T(t,x)\,.
\end{split}
\ee
 It suffices to estimate the first and the last terms, as the remaining ones follow by similar arguments. \par 
Integrating by parts, we can rewrite the first integral as follows
\begin{align*}
   -\int_0^t \int_{\mathbb{R}^n_+}
\n_xG_{ij}(t-\tau,x,y)D_{y_p}a_pb_j&(y,\tau)\,dy\,d\tau= \int_0^{\frac t2} \int_{\mathbb{R}^n_+}
\n_x D_{y_p}G_{ij}(t-\tau,x,y)a_pb_j(y,\tau)\,dy\,d\tau\\
&
-\int_{\frac t2}^t \int_{\mathbb{R}^n_+}
\n_xG_{ij}(t-\tau,x,y)a_pD_{y_p}b_j(y,\tau)\,dy\,d\tau=: I_1+I_2.
\end{align*}
Considering the $L^q$-norm of $I_1$, taking $r=q'$ in estimate \eqref{StimeyoungKernellx}, with $|h|=|l|=1$ and $k=0$, applying the Young inequality and using hypothesis on $a$ and $b$, we obtain
\begin{equation}\label{prima}
\begin{split}
\vspace{5pt}
    \dm I_1\dm_q \leq  c\!\!\int_0^{\frac t2} \frac{\dm a(\tau) \dm_q \dm b(\tau)\dm_q}{(t-\tau)^{1+\frac{n}{2q}}}\,d\tau &\leq c \sup_{(0,T)} \tau^{\frac{n}{2} \left( \frac{1}{n}-\frac{1}{q}\right)} \dm a(\tau) \dm_{q} \,\sup_{(0,T)} \tau^{\frac{n}{2} \left( \frac{1}{n}-\frac{1}{q}\right)} \dm b(\tau) \dm_{q} \!\int_0^{\frac t2}\!\! \frac{1}{\tau^{1-\frac{n}{q}}(t-\tau)^{1+\frac{n}{2q}}}\, d\tau \\
    \vspace{5pt}
    &\leq c\, t^{-\frac 12 -\frac n2(\frac 1n -\frac 1q)}\sup_{(0,T)} \tau^{\frac{n}{2} \left( \frac{1}{n}-\frac{1}{q}\right)} \dm a(\tau) \dm_{q} \,\sup_{(0,T)} \tau^{\frac{n}{2} \left( \frac{1}{n}-\frac{1}{q}\right)} \dm b(\tau) \dm_{q}\,.    
\end{split}
\ee
Similarly, for $I_2$ thanks to estimate \eqref{StimeyoungKernellx}, with $|h|=1$ and $|l|=k=0$, we obtain
\begin{equation}\label{seconda}
\begin{split}
    \dm I_2\dm_q &\leq c \int_{\frac t2}^t
\frac{\dm a(\tau)\dm_q\,\dm \n b(\tau)\dm_q}{(t-\tau)^{1+\frac{n}{2q}}}\,d\tau\\
& \leq  c \sup_{(0,T)} \tau^{\frac{n}{2} \left( \frac{1}{n}-\frac{1}{q}\right)} \dm a(\tau) \dm_{q} \,\sup_{(0,T)} \tau^{\frac 12+\frac{n}{2} \left( \frac{1}{n}-\frac{1}{q}\right)} \dm \n b(\tau) \dm_{q} \int_{\frac t2}^t \frac{1}{\tau^{\frac{3}{2}-\frac{n}{q}}(t-\tau)^{1+\frac{n}{2q}}}\, d\tau \\
    &\leq c\, t^{-\frac 12 -\frac n2(\frac 1n -\frac 1q)}\sup_{(0,T)} \tau^{\frac{n}{2} \left( \frac{1}{n}-\frac{1}{q}\right)} \dm a(\tau) \dm_{q} \,\sup_{(0,T)} \tau^{\frac 12 + \frac{n}{2} \left( \frac{1}{n}-\frac{1}{q}\right)} \dm \n b(\tau) \dm_{q}\,.
    \end{split}
\ee
Concerning the last term in \eqref{rapfabmodgrad}, considering, by estimate \eqref{stimeLrKx} with $|h|=3$ and $|l|=0$, for a general term of the linear combination, we obtain
\small{\begin{align*}
    &\left | \!\left| \int_0^t  \!\!\!\int_{R^n_+}\!\!\!\!\!\! D_{x_n x_s x_m} K_{ijq} (t\! -\!\tau,x,z) a_p(\tau,z)b_r(\tau,z) \, dz 
   \right | \!\right|_q \\
   &\leq c \sup_{(0,T)} \tau^{\frac{n}{2} \left( \frac{1}{n}-\frac{1}{q}\right)} \dm a(\tau) \dm_{q} \,\sup_{(0,T)} \tau^{\frac{n}{2} \left( \frac{1}{n}-\frac{1}{q}\right)} \dm b(\tau) \dm_{q} \int_0^{t} \ssfrac{1}{\tau^{1-\frac{n}{q}}(t-\tau)^{1+\frac{n}{2q}}}\, d\tau \\
   &\leq c\, t^{-\frac 12 -\frac n2(\frac 1n -\frac 1q)}\sup_{(0,T)} \tau^{\frac{n}{2} \left( \frac{1}{n}-\frac{1}{q}\right)} \dm a(\tau) \dm_{q} \,\sup_{(0,T)} \tau^{\frac{n}{2} \left( \frac{1}{n}-\frac{1}{q}\right)} \dm b(\tau) \dm_{q}\,. 
\end{align*}}
Finally, summing the relations just obtained and estimate \eqref{prima}, \eqref{seconda}, we deduce the thesis.
\end{proof}
\begin{lemma}\label{manca}
    \sl{
    Assume that
\[\sup_{(0,t)}\tau^{\frac{1}{2}}\dm a(\tau)\dm_{\infty} +  \sup_{(0,T)} \tau^{\frac{n}{2} \left( \frac{1}{n}-\frac{1}{q}\right)} \dm b(\tau) \dm_{L^q(\R^n_+)}\,.\]
Then, there exists a $c$ independent of $a(t,x)$ and $b(t,x)$ such that
\be\label{Stimelinfpnl}
t^{\frac{1}{2}}\dm \ov{u}(t)\dm_{\infty}\leq c \sup_{(0,t)}\tau^{\frac{1}{2}}\dm a(\tau)\dm_{\infty}\,\sup_{(0,T)} \tau^{\frac{n}{2} \left( \frac{1}{n}-\frac{1}{q}\right)} \dm b(\tau) \dm_{L^q(\R^n_+)}\,,
\ee
for all $t\in (0,T)$.
    }
\end{lemma}
\begin{proof}
    For the sake of brevity, we omit the proof, as it is analogous to that of the following lemma, with a trivial adaptation.
\end{proof}
\begin{lemma}\label{PPNLF}
{\sl
Assume that
\[\sup_{(0,t)}\tau^{\frac{1}{2}}\dm a(\tau)\dm_{\infty}+\sup_{(0,t)}\tau^{\frac{n}{2p}}\sup_{\R^n_+}|y|^{\alpha}|b(\tau,y)|< \infty.\]
Then, there exists a constant $c$, independent of $a$ and $b$, such that
\begin{equation}\label{sppnl}
\begin{array}{cc}
     &|x|^{\alpha}\,t^{\frac{n}{2p}}|\ov{u}(t,x)|\leq c\, \dy \sup_{(0,t)}\tau^{\frac{1}{2}}\dm a(\tau)\dm_{\infty}\, \dy\sup_{(0,t)}\tau^{\frac{n}{2p}} \dy \sup_{\R^n_+}|y|^{\alpha}|b(\tau,y)|\,,  \\
     & |x|^{2\alpha}\,t^{\frac 12 -\frac{n}{p}}|\ov{u}(t,x)|\leq c\,\big[\dy\sup_{(0,t)}\tau^{\frac{n}{2p}} \dy\sup_{\R^n_+}|y|^{\alpha}|b(\tau,y)|\big]^2\,.
\end{array}
\end{equation}
Moreover, assume that $a\equiv b$ and
\begin{equation*}
    \sup_{(0,t)}\tau^{\frac{n}{2p}}\sup_{\R^n_+}|y|^{\alpha}|b(\tau,y)|= o(1)\quad \mbox{in } t=0\,,
\end{equation*}
then one also gets
\begin{equation}\label{infinitesimopnl}
    |x|^{2\alpha}\,t^{\frac 12 -\frac{n}{p}}|\ov{u}(t,x)|=o(1) \quad \mbox{in } t=0\,, \mbox{ for all }x\in \R^n_+\,.
\end{equation}
}
\end{lemma}
\begin{proof}
Let us consider the representation \eqref{rapfabmod}. From pointwise estimates \eqref{SKIC} and \eqref{SKKS} on the integral kernels, we have
  \begin{align}
      |\bar{u}(x,t)| & \leq c\int_0^t \int_{\R^n_+}\frac{|a(\tau,y)| |b(\tau, y)|}{(|x-y|^2 + (t-\tau))^{\frac{n+1}{2}}} \, dy d\tau + c\int_0^t \int_{\R^n_+}\frac{|a(\tau,y)| |b(\tau, y)| e^{-\frac{\hat{a}y_n^2}{(t-\tau)}}}{(|x-y^*|^2 + (t-\tau))^{\frac{n}{2}}(t-\tau)^{\frac{1}{2}}} \, dy d\tau \nonumber\\
      & := J_1 + J_2, \label{d}
  \end{align} 
and we set
\[D:=\sup_{(0,t)}\tau^{\frac{1}{2}}\dm a(\tau)\dm_{\infty}\,\sup_{(0,t)}\tau^{\frac{n}{2p}}\sup_{\R^n_+}|y|^{\alpha}|b(\tau,y)|\,.\]
Concerning $J_1$, we have
\begin{align*}
 J_1 & \leq \!\int_0^t\!\int_{|y|<\frac{|x|}{2}}\!\!\sfrac{|a(\tau,y)||b(\tau,y)|}{(|x-y|^2+(t-\tau))^{\frac{n+1}{2}}}\,dy\,d\tau
+  \!\int_0^t\!\int_{|y|>\frac{|x|}{2}}\!\!\sfrac{|a(\tau,y)||b(\tau,y)|}{(|x-y|^2+(t-\tau))^{\frac{n+1}{2}}}\,dy\,d\tau\\&=:J^1_1+J^2_1\,.
\end{align*}
Since $|x-y|\geq |x| -|y|\geq \frac{|x|}{2}$ for $|y|<\frac{|x|}{2}$, then  for $J^1_1$ we get
\begin{align*}
J^1_1\leq {D}{|x|^{-n}}\int_0^t \tau^{-\frac{1}{2}-\frac{n}{2p}}(t-\tau)^{-\frac{1}{2}}\int_{|y|<\frac{|x|}{2}}{|y|^{-\alpha}}\,dy\,d\tau\leq {c\,D}{|x|^{-\alpha}\, t^{-\frac{n}{2p}}}.
\end{align*}
For $J^2_1$ we have
\begin{align*}
J^2_1 &\leq {c\, D}{|x|^{-\alpha}}\int_0^t {\tau^{-\frac{1}{2}-\frac{n}{2p}}}\bigg(\int_{|y|>\frac{|x|}{2}}{(|x-y|+(t-\tau)^{\frac{1}{2}})^{-n-1}}\,dy\bigg)\,d\tau\VSE\leq {c\, D}{|x|^{-\alpha}}\int_0^t {\tau^{-\frac{1}{2}-\frac{n}{2p}}(t-\tau)^{-\frac{1}{2}}}\,d\tau\leq \frac{c\, D}{|x|^{\alpha}\, t^{\frac{n}{2p}}}\,.
\end{align*}
So we arrive at
\begin{equation} \label{stimaJ_1}
J_1 \leq \frac{cD}{|x|^\alpha t^{\frac{n}{2p}}}.
\end{equation}
Analogously, for $J_2$ we get
\begin{align}
    J_2 & \leq cD \int_0^t \frac{1}{\tau^{\frac{1}{2} + \frac{n}{2p}}} \int_{\R^n_+} \frac{ e^{-\frac{\hat{a}y_n^2}{(t-\tau)}}}{(|x-y^*|^2 + (t-\tau))^{\frac{n}{2}}(t-\tau)^{\frac{1}{2}}} \frac{1}{|y|^\alpha} \, dyd\tau \nonumber\\
    &\leq cD \int_0^t \frac{1}{\tau^{\frac{1}{2} + \frac{n}{2p}}} \int_{|y|<\frac{|x|}{2}} \frac{ e^{-\frac{\hat{a}y_n^2}{(t-\tau)}}}{(|x-y^*|^2 + (t-\tau))^{\frac{n}{2}}(t-\tau)^{\frac{1}{2}}} \frac{1}{|y|^\alpha} dyd\tau \nonumber\\
    & \hspace{0.45cm} + cD \int_0^t \frac{1}{\tau^{\frac{1}{2} + \frac{n}{2p}}}\int_{|y|>\frac{|x|}{2}} \frac{ e^{-\frac{\hat{a}y_n^2}{(t-\tau)}}}{(|x-y^*|^2 + (t-\tau))^{\frac{n}{2}}(t-\tau)^{\frac{1}{2}}} \frac{1}{|y|^\alpha} dyd\tau \nonumber\\
    &=:\int_0^t \frac{1}{\tau^{\frac{1}{2} + \frac{n}{2p}}} \left( J_2^1 + J_2^2\right) \, d\tau. \label{e}
\end{align}
For $J_2^1$, since $|y|=|y^*|$, we have $|x-y^*|>\frac{|x|}{2}$, we observe
\begin{equation}\label{stimaJ_2^1}
    J_2^1 \leq \frac{1}{|x|^n} \int_{|y|<\frac{|x|}{2}} \frac{1}{(t-\tau)^{\frac{1}{2}}} \frac{1}{|y|^\alpha} \, dy \leq \frac{c}{|x|^\alpha (t-\tau)^{\frac{1}{2}}}.
\end{equation}
Moreover for $J_2^2$ by simple computations, we have
\begin{align}
  J_2^2 & \leq \frac{1}{|x|^\alpha}\int_0^{+\infty} e^{-\frac{\hat{a}y_n^2}{(t-\tau)}} \int_{\R^{n-1}}\frac{1 }{(|x'-y'|^2 +x_n^2 + y_n^2 + (t-\tau))^{\frac{n}{2}}(t-\tau)^{\frac{1}{2}}} \, dy' dy_n \nonumber \\
  & \leq \frac{c}{|x|^\alpha (t-\tau)^\frac{1}{2}}\int_0^{+\infty}  \frac{e^{-\frac{\hat{a}y_n^2}{(t-\tau)}} }{(x_n^2 + y_n^2 + (t-\tau))^{\frac{1}{2}}}\, dy_n \leq \frac{c}{|x|^\alpha (t-\tau)^\frac{1}{2}}. \label{stimaJ_2^2}.
\end{align}
Now, using \eqref{stimaJ_2^1} and \eqref{stimaJ_2^2} in \eqref{e}, we get
\begin{equation}
    J_2 \leq \frac{cD}{|x|^\alpha} \int_0^t \frac{1}{\tau^{\frac{1}{2} + \frac{n}{2p}}(t-\tau)^\frac{1}{2}} \, d\tau \leq \frac{cD}{t^{\frac{n}{2p}}|x|^\alpha}. \label{stimaJ_2}.
\end{equation}
Estimates \eqref{stimaJ_1} and \eqref{stimaJ_2} used in \eqref{d}, recalling the definition of $D$, ensure the estimate \eqref{sppnl}$_1$. Following the same arguments line, one easly get  \eqref{sppnl}$_2$. Then, we limit ourselves to prove \eqref{infinitesimopnl}.\par
Assume that $a \equiv b$. As before, we have
  \begin{align}
      |\bar{u}(x,t)| & \leq c\int_0^t \int_{\R^n_+}\ssfrac{|b(\tau, y)|^2}{(|x-y|^2 + (t-\tau))^{\frac{n+1}{2}}} \, dy d\tau + c\int_0^t \int_{\R^n_+}\ssfrac{|b(\tau, y)|^2 \,e^{-\frac{\hat{a}y_n^2}{(t-\tau)}}}{(|x-y^*|^2 + (t-\tau))^{\frac{n}{2}}(t-\tau)^{\frac{1}{2}}} \, dy d\tau \nonumber := I_1 + I_2\,.\label{f}
  \end{align} 
By hypothesis, for all $\varepsilon >0$, there exists an instant $\ov{t}$ such that, for all $t\in (0, \ov{t})$
\begin{align*}
 I_1 & \leq \!\int_0^t\!\int_{|y|<\frac{|x|}{2}}\!\!\ssfrac{|b(\tau,y)|^2}{(|x-y|^2+(t-\tau))^{\frac{n+1}{2}}}\,dy\,d\tau
+  \!\int_0^t\!\int_{|y|>\frac{|x|}{2}}\!\!\ssfrac{|b(\tau,y)|^2}{(|x-y|^2+(t-\tau))^{\frac{n+1}{2}}}\,dy\,d\tau \\
&\leq 
 c\, \varepsilon^2 \bigg( 
{|x|^{-n}}\int_0^t \tau^{\frac{n}{p}}(t-\tau)^{-\frac{1}{2}}\int_{|y|<\frac{|x|}{2}}{|y|^{-2\alpha}}\,dy\,d\tau + 
{|x|^{-2\alpha}}\int_0^t {\tau^{\frac{n}{p}}(t-\tau)^{-\frac{1}{2}}}\,d\tau\bigg)\\
&\leq c\, \varepsilon ^2 {|x|^{-2\alpha}\, t^{\frac 12 -\frac{n}{p}}}\,.
\end{align*}
To estimate $I_2$, it is enough to argue as above and proceed as for $J_2$. Thus, the estimate \eqref{infinitesimopnl} follows. 
\end{proof}
\section{Proof of the main results}\label{sect. 5}
We are in a position to establish the main results of our work, namely the existence and uniqueness theorems for the problem under consideration.\\ As in \cite{MP}, the strategy relies on the construction of approximate solutions through an iterative scheme based on the integral formulation of the equation.\par
We set 
\be
    \mathcal{G}[h] := \int_{\R^n_+} \mathcal{G}(t,x,y)h(y)\,dy; \label{G}
\ee
\be
\mathcal{S}[g \cdot \n g] := \int_0^t \int_{\mathbb{R}^n_+}
\mathcal{G}(t-\tau,x,y)\,P [g(\tau,y) \cdot \n g(\tau,y)]\,dy\,d\tau \label{S}.
\ee

For $m\in \N$, we consider the following approximation scheme:
\begin{equation}\label{definizione_schema_iterativo}
u^m(t,x)= \mathcal{G}[u_0] - \mathcal{S}[u^{m-1} \cdot \n u^{m-1}]\,.
\end{equation}
where $u^{m-1}\equiv 0$ for $m=0$, and $u^0=\mathcal{G}[u_0]$ is the solution given in Theorem$\,$\ref{Theoremprincipale}.\par
We set
\begin{equation}\label{definizione_funzionale}
\widetilde{\tm} u(t)\tm\!\!:= \!\sup_{(0,t)} \tau^{\frac{1}{2}}\dm u(\tau)\dm _{\infty}\! + \sup_{(0,t)}\tau^{\frac{n}{2}(\frac{1}{n}-\frac{1}{q})}\dm u(\tau)\dm _q\!+\sup_{(0,t)}\sup_{\R^n}\tau^{\frac{n}{2p}}|y|^{\alpha}|u(\tau,y)|\! + \dy \sup_{(0,t)}\tau^{\frac{n}{2} \left( \frac{1}{n}-\frac{1}{q} \right) +\frac{1}{2}} \dm \n u(\tau) \dm_{q},
\end{equation}
and in the following, for the sake of readability, we just write $K(t,\rho)$ in place of $K_{u_0}(t,\rho)$ defined in \eqref{krho}.
\begin{lemma}\label{lemmaiterativo}{\sl 
Assume that $u_0 \in L^p_{\texttt{w}}(\R^n_+)$. Then there exist constants $c,c_2>0$, independent of $u_0$, such that for any element of the sequence $\lbrace u^m \rbrace_{m\in \N}$ defined as in \eqref{definizione_schema_iterativo}, the following estimate holds for every $T>0$:
\begin{equation}\label{stimaiter}
\widetilde{\tm} u^m(t)\tm\leq K(t,\rho) + c\,\dm u_0\dm _{\texttt{w},p}^{1-\gamma}K(t,\rho)^{\gamma}+ c_2\,\widetilde{\tm} u^{m-1}(t)\tm^{2}\,, \quad \text{for all } t\in(0,T)\, \text{and }m \in \N\,,
\end{equation}
where $\gamma=1-\frac{n}{q}$.
}
\end{lemma}
\begin{proof}
Let $m=1$. From definition \eqref{definizione_schema_iterativo}, by virtue of Lemma\,\ref{piclinf} and Lemma\,\ref{manca}, for all $s>0$ and $\rho>0$, we get
\begin{align*}
s^{\frac{1}{2}}\dm u^1(s)\dm _{\infty}& \leq s^\frac{1}{2} \dm \mathcal{G}[u_0](s)\dm_\infty + s^\frac{1}{2} \dm \mathcal{S}[u_0 \cdot \n u_0 ](s)\dm_\infty \\
&\leq s^{\frac{1}{2}}\dm u^0 (s)\dm _{\infty} + c\,\Big(\sup_{(0,s)} \tau^{\frac{1}{2}}\dm u^0(\tau)\dm _{\infty}+\sup_{(0,s)}\tau^{\frac{n}{2}(\frac{1}{n}-\frac{1}{q})}\dm u^0(\tau)\dm _q + \sup_{(0,s)}\tau^{\frac 12+\frac{n}{2}(\frac{1}{n}-\frac{1}{q})}\dm \n u^0(\tau)\dm _q\Big)^2 \\
&\leq K(s,\rho) + c\widetilde{\tm} u^0(s)\tm^2\,.
\end{align*}
Hence, being the right-hand side   an increasing function of $s$, on any interval $(0,t)$, we get 
\be\sup_{(0,t)} s^\frac12\dm u^1(s)\dm_\infty\leq K(t,\rho)+c\widetilde{\tm}u^0(t)\tm^2\,. \label{infi}\ee
We claim that the following estimate hold
$$s^{\frac{n}{2}(\frac{1}{n}-\frac{1}{q})}\dm u^1(s)\dm_q \leq s^{\frac{n}{2}(\frac{1}{n}-\frac{1}{q})} \dm \mathcal{G}[u_0](s)\dm_q + s^{\frac{n}{2}(\frac{1}{n}-\frac{1}{q})} \dm \mathcal{S}[u_0 \cdot \n u_0 ](s)\dm_q. $$
Actually, applying Lemma \ref{interpol}, together with Lemma \ref{piclinf}, we have
\be s^{\frac{n}{2}(\frac{1}{n}-\frac{1}{q})} \dm \mathcal{G}[u_0]\dm_q \leq c s^{\frac{n}{2}(\frac{1}{n}-\frac{1}{q})}  \dm \mathcal{G}[u_0]\dm^{1-\gamma}_{\texttt{w},p} \dm \mathcal{G}[u_0]\dm_\infty^{\gamma} \leq c \,\dm u_0\dm^{1-\gamma}_{\texttt{w},p} K (s, \rho)^{\gamma}\label{ql}\,. \ee
Now, Lemma \ref{stimaq} ensures
\be s^{\frac{n}{2}(\frac{1}{n}-\frac{1}{q})} \dm \mathcal{S}[u_0 \cdot \n u_0 ]\dm_q  \leq c \widetilde{\tm} u_0 (s) \tm^2\,, \label{qnl}\ee
and summing this estimate to \eqref{ql}, we obtain
\begin{align*}
s^{\frac{n}{2}(\frac{1}{n}-\frac{1}{q})}\dm u^1(s)\dm _q  \leq c\dm u_0\dm _{\texttt{w},p}^{1-\gamma}K(s,\rho)^{\gamma} + c\,\widetilde{\tm} u^0(s)\tm^2\,.
\end{align*}
The same argument lines employed for $L^\infty$-estimate lead to the following:
\be \sup_{ (0,t)}s^{\frac{n}{2}(\frac{1}{n}-\frac{1}{q})}\dm u^1(s)\dm _q  \leq c\dm u_0\dm _{\texttt{w},p}^{1-\gamma}K(t,\rho)^{\gamma} + c\,\widetilde{\tm} u^0(t)\tm^2\,.
\label{normq}
\ee 
Moreover, taking in account Corollary \ref{andamenti_asintotici} for the linear part and the estimate \eqref{sppnl}$_1$ for the non linear term, we get
\begin{align*}
s^{\frac{n}{2p}}|x|^{\alpha}|u^1(s,x)|\leq s^{\frac{n}{2p}}|x|^{\alpha}|u^0(s,x)| + c\sup_{(0,s)}\tau^{\frac{1}{2}}\dm u^0(\tau)\dm_{\infty}\sup_{(0,s)}\tau^{\frac{n}{2p}}\dm|y|^{\alpha}u^0(\tau)\dm_\infty\leq K(s,\rho) + c \,\widetilde{\tm} u^0(s) \tm^{2}\,.
\end{align*}
Hence, we get
\be \sup_{(0,t)}s^{\frac{n}{2p}}|x|^{\alpha}|u^1(s,x)|\leq K(t,\rho) + c \,\widetilde{\tm} u^0(t) \tm^{2}\,. \label{punt} \ee
Finally, for all $s \in (0,t)$, we have
$$\dy s^{\frac{n}{2} \left( \frac{1}{n}-\frac{1}{q} \right) +\frac{1}{2}} \dm \n u^1(s) \dm_{L^q(\R^n_+)} \leq s^{\frac{n}{2} \left( \frac{1}{n}-\frac{1}{q} \right) +\frac{1}{2}} \dm \n \mathcal{G}[u_0](s) \dm_{L^q(\R^n_+)} + s^{\frac{n}{2} \left( \frac{1}{n}-\frac{1}{q} \right) +\frac{1}{2}} \dm \n  \mathcal{S}[u_0](s) \dm_{L^q(\R^n_+)}\,. $$
Now, Lemma \ref{LemmapiccLq} implies 
\be  \label{GLQLN}
s^{\frac{n}{2} \left( \frac{1}{n}-\frac{1}{q} \right) +\frac{1}{2}} \dm \n \mathcal{G}[u_0](s) \dm_{L^q(\R^n_+)} \leq \dm u_0 \dm_{\texttt{w},p}^{1-\gamma} K(s, \rho)^\gamma,
\ee
and Lemma \ref{LEmmaPNLGRADLq}
\be  \label{GLQNL}
s^{\frac{n}{2} \left( \frac{1}{n}-\frac{1}{q} \right) +\frac{1}{2}} \dm \n \mathcal{S}[u_0 \cdot \n u_0](s) \dm_{L^q(\R^n_+)} \leq \dm u_0 \dm_{\texttt{w},p}^{1-\gamma} K(s, \rho)^\gamma +  c \tm u_0 \tm^2.
\ee
Thus, we have
\be  \label{GLQ}
\dy \sup_{(0,t)}s^{\frac{n}{2} \left( \frac{1}{n}-\frac{1}{q} \right) +\frac{1}{2}} \dm \n u^1(s) \dm_{L^q(\R^n_+)} \leq c \dm u_0 \dm_{\texttt{w},p}^{1-\gamma} K(s, \rho)^\gamma+ c \tm u_0 \tm^2.
\ee
Summing \eqref{infi}, \eqref{normq}, \eqref{punt}, \eqref{GLQ} we obtain
\begin{equation}\label{approssimazione_funzionale}
\widetilde{\tm} u^1(s)\tm\leq K(s,\rho)+  c\,\dm u_0\dm _{\texttt{w},p}^{1-\gamma}K(s,\rho)^{\gamma} + 3c\, \widetilde{\tm} u^0(s)\tm^{2},
\end{equation}
with constant $c$ independent of the datum $u_0$.\\
Now, let $m>1$ and let assume that $u^m$ satisfies the property \eqref{stimaiter}. Following the same line as above, since $u^m$ satisfies the hypothesis of Lemma\,\ref{piclinf} -\,\ref{PPNLF}, property \eqref{stimaiter} holds for all $m\in \N$.
\end{proof}
\begin{lemma}\label{lemmanumerico}{\sl 
Let $\xi_0 >0$ and $c>0$. Let $\{\xi_m\}$ be a nonnegative sequence of real numbers such that:
\begin{equation*}
\xi_m\leq \xi_0+c\,\xi_{m-1}^2.
\end{equation*}
Assume $1-4\xi_0>0$ and  $\xi_0\leq \xi$, where $\xi$ is the minimum solution of algebraic equation $c\xi^2-\xi+\xi_0=0$. Then $\xi_{m-1}\leq \xi$ for all $m\in \N$.}
\end{lemma}
\begin{proof}
For the proof we remind to \cite{solo}, Lemma 10.2.
\end{proof}
\begin{lemma}\label{convergenza}
  \sl{ Let $\{u^m\}$ be the sequence defined in \eqref{definizione_schema_iterativo} corrisponding to $u_0\in L^p_{\texttt{w}}(\R^n_+)$. Then, there exists $T(u_0)>0$ such that, for all $\eta>0$, the sequence strongly converges in $C((\eta,T(u_0))\times\R^n_+)$ to a solution $u$ of \eqref{Soluzione}. In particular, for a suitable $\rho$ and for all $t\in [0,T(u_0))$, we get
\begin{equation}\label{A(ro,t)_enunciato}
\widetilde{\tm} u(t)\tm  \leq \sfrac{2\left[K(t,\rho)+  c\dm u_0\dm _{\texttt{w},p}^{1-\gamma}K(t,\rho)^{\gamma}\right]}{1+\left[1-4c_2\,\left(K(t,\rho)+  c\dm u_0\dm _{\texttt{w},p}^{1-\gamma}K(t,\rho)^{\gamma}\right)\right]^{\frac{
1}{2}}}\leq \frac{1}{4c}
\end{equation}
and, in particular
\begin{gather}\label{andamento_lemma}
t^\frac n{2p}\dm u(t)|x|^\alpha\dm_\infty+t^{\frac{1}{2}}\dm u(t)\dm _{\infty}+t^{\frac{n}{2}(\frac{1}{n}-\frac{1}{q})}\dm u(t)\dm _q + t^{\frac{n}{2}(\frac{1}{n}-\frac{1}{q})+\frac{1}{2}}\dm \n u(t)\dm _q \leq K(t,\rho)+  c\dm u_0\dm _{\alpha,p}^{1-\gamma}K(t,\rho)^{\gamma},
\end{gather}
where $\gamma=1-\frac{n}{q}$. \par
Moreover, the following limit property holds
\begin{equation}\label{limit property_lemma}
    \lim_{t \to 0^+} \tm u(t)\tm=0\,.
\end{equation}
}
\end{lemma}
\begin{proof}
Let $u_0 \in L^p_{\texttt{w}}(\R^n)$ and fix $\varepsilon \in \bigl(0,\frac{1}{4c_2 c_0}\bigr)$. By definition of $\dm \cdot \dm_{\texttt{w},p,\rho}$ in \eqref{normarho}, there exists $\rho=\rho(u_0,\varepsilon)$ such that $\dm u_0\dm_{\texttt{w},p,\rho}<\varepsilon$. \\
For such a choice of $\rho$, we consider
\begin{equation}\label{istante_sup_esistenza}
t(\rho):= \displaystyle \sup\Big\{t>0:\ 1-4c_2\Big(K(t,\rho)+c\,\dm u_0\dm_{\texttt{w},p}^{1-\gamma}K(t,\rho)^{\gamma}\Big)>0\Big\},
\end{equation}
We define $ T(u_0):= \dy \sup_{\rho>0}t(\rho)$.

We first derive uniform bounds for the approximating sequence. Combining \eqref{stimaiter} with Lemma~\ref{lemmanumerico}, it follows that for any fixed $\rho$ and all $t\in[0,T(u_0))$,
\begin{equation}\label{A(ro,t)}
\widetilde{\tm} u^m(t)\tm  \leq \sfrac{2\left[K(t,\rho)+  c\dm u_0\dm _{\texttt{w},p}^{1-\gamma}K(t,\rho)^{\gamma}\right]}{1+\left[1-4c_2\,\left(K(t,\rho)+  c\dm u_0\dm _{\texttt{w},p}^{1-\gamma}K(t,\rho)^{\gamma}\right)\right]^{\frac{
1}{2}}}=:A(\rho,t).
\end{equation}
In particular, $\{\tm u^m\tm\}$ is uniformly bounded on $[0,T(u_0))$.\par
To prove convergence, it is convenient to work with the sequence $w^m:=u^m-u^{m-1}$. A direct computation based on \eqref{definizione_schema_iterativo} shows that
\begin{equation*}
w^{m+1}= \int_0^t \int_{\mathbb{R}^n_+}
\mathcal{G}(t-\tau,x,y)\,P [w^m \cdot \n u^m((y,\tau)]\,dy\,d\tau - \int_0^t \int_{\mathbb{R}^n_+}
\mathcal{G}(t-\tau,x,y)\,P [u^{m-1} \cdot \n w^m((y,\tau)]\,dy\,d\tau.
\end{equation*}
Using Lemmas~\ref{stimaq}--\ref{PPNLF} together with the bound \eqref{A(ro,t)}, we obtain 
\begin{equation}\label{assoluta_convergenza}
\widetilde{\tm} w^m(t)\tm \leq C\,\big(cA(\rho,t)\big)^m A(\rho,t),
\end{equation}
for a suitable constant $C>0$. By definition of $T(u_0)$, we have $A(\rho,t)<\frac{1}{2c}$ for $t\in(0,T(u_0))$, hence the series is convergent. Therefore, writing
\begin{equation}\label{series}
u^m=u^0+\sum_{i=1}^m w^i,
\end{equation}
we deduce that $\{u^m\}$ is a Cauchy sequence with respect to $\widetilde{\tm}\cdot\tm$ and thus converges to some limit function $u$ that, in particular, enjoys of estimates \eqref{A(ro,t)_enunciato} and \eqref{andamento_lemma}. In particular, these estimates ensure that the nonlinear term $\int_0^t \int_{\mathbb{R}^n_+}
\mathcal{G}(t-\tau,x,y)\,P [u \cdot \n u((y,\tau)]\,dy\,d\tau$ is well defined. Passing to the limit in the integral formulation (using dominated convergence, justified by the uniform bound \eqref{A(ro,t)}), we obtain
\begin{equation}\label{rappresentazione_integrale_soluzione}
u(t,x)=\mathcal{G}[u_0](t,x)-\mathcal{S}[(u\cdot \n u)](t,x).
\end{equation}
Moreover, since the convergence is uniform on compact subsets of $(\eta,T)\times\R^n$ and each $u^m$ is continuous, the limit $u$ is continuous as well.
We now analyze the property \eqref{limit property_lemma}. From \cite{CM-L3}, the following property holds:
\begin{itemize}
\item[(P)] for every sequence $t_p\to 0^+$, one can find $\rho_p\to 0^+$ such that
\[
1-4c_2\Big(K(t_p,\rho_p)+c\,\dm u_0\dm_{\texttt{w},p}\Big)>0,
\quad \text{and} \quad A(\rho_p,t_p)\to 0.
\]
\end{itemize}
Applying Lemma~\ref{lemmanumerico} along these sequences, we deduce
\begin{equation}\label{stima(P)}
\widetilde{\tm} u^m(t_p)\tm \leq A(\rho_p,t_p).
\end{equation}
Passing to the limit and recalling the definition of the functional $\tm\cdot\tm$, we obtain \eqref{limit property_lemma}.
\end{proof}
\begin{lemma}\label{andamento_a_0_parte_non_lineare}{\sl Let $u$ a vector field enjoying \eqref{andamento_sol_q} and $\ov u$ defined by formula \eqref{ra}.  \begin{itemize}\item
If $p \in (n,2n)$, then, for all $q\in (n,2n)$, we get 
\begin{equation}\label{prop_L_parte_non_lineare}
\dm \overline{u}(t)\dm _{\frac{q}{2}}\leq c(u_0)\, t^{-\frac12 +\frac{n}{q}}\qquad \mbox{for all }t\in(0,T)\,;
\end{equation}
\item If $p=2n$, then we get
\begin{equation}\label{proprieta_L_weak_parta_non_lineare1}
\dm u(t)\dm_{(n,\infty)}<\infty\,,\;t\in(0,T), \; \; \mbox{ and } \; \; \lim_{t\to 0^+}\dm\overline{u}(t)\dm _{(n,\infty)}=0\,;
\end{equation}
\item Finally, if $p>2n$, then we get
\begin{equation}\label{proprieta_L_weak_parta_non_lineare}
\dm \overline{u}(t)\dm _{(\frac{np}{2(p-n)},\infty)}\leq {c(u_0)}\,t^{\frac{1}{2}-\frac{n}{p}}\qquad \mbox{for all }t\in(0,T).
\end{equation}
\end{itemize} }
\end{lemma}
\begin{proof}
Let $p\in (n,2n)$ and fix $q\in(n,2n)$. By estimate \eqref{andamento_sol_q}, the solution satisfies
\[
\dm u(t)\dm_q \leq c\, t^{-\frac{1}{2}+\frac{n}{2q}} \dm u_0\dm_{\texttt{w},p}.
\]
We now turn to the nonlinear contribution. By combining Minkowski's inequality with Young inequality, we have
\[
\dm \overline{u}(t)\dm_{\frac{q}{2}}
\leq c \int_0^t \dm \nabla G(t-\tau)\dm_1 \, \dm u(\tau)\dm_q^2 \, d\tau.
\]
Substituting the previous bound on $\dm u(\tau)\dm_q$ into the integral, we obtain
\[
\dm \overline{u}(t)\dm_{\frac{q}{2}}
\leq c\, \dm u_0\dm_{\texttt{w},p}^2 \int_0^t \tau^{-1+\frac{n}{q}} (t-\tau)^{-\frac{1}{2}}\, d\tau.
\]
A direct evaluation of the time integral yields
\[
\dm \overline{u}(t)\dm_{\frac{q}{2}} \leq c(u_0)\, t^{-\frac{1}{2}+\frac{n}{q}}.
\]
Consider now the case $p=2n$. Exploiting the pointwise bound given by \eqref{andamento_sol_q}, we estimate the nonlinear term as follows:
\[
|\overline{u}(t,x)|
\leq \frac{c(\dm u_0\dm_{\texttt{w},p})}{|x|}
\int_0^t \tau^{-\frac{1}{2}}(t-\tau)^{-\frac{1}{2}}\, d\tau
\leq c(u_0)\,|x|^{-1},
\]
for all $t\in(0,T)$. This immediately yields the uniform bound
\[
\dm \ov{u}(t)\dm_{(n,\infty)} \leq c(u_0), \qquad t\in(0,T).
\]
We next investigate the behavior as $t\to 0^+$. Fix $\varepsilon>0$ and let $\{t_m\}$ be a sequence with $t_m\to 0^+$. By \eqref{limite_linfinito_soluzione}, there exists $\overline{m}$ such that for all $m>\overline{m}$,
\[
|u(t_m,x)|
\leq \frac{c(\dm u_0\dm_{\texttt{w},p})\varepsilon}{t_m^{\frac{1}{4}}|x|^{\frac{1}{2}}}
.
\]
Inserting this estimate into the representation of $\overline{u}$ and arguing as above, we obtain
\[
|\overline{u}(t_m,x)|
\leq \frac{c(\dm u_0\dm_{\texttt{w},p})\varepsilon}{t_m^{\frac{1}{4}}|x|^{\frac{1}{2}}} \int_0^{t_m} \tau^{-\frac{1}{2}}(t_m-\tau)^{-\frac{1}{2}}\, d\tau
\leq c(u_0)\,\varepsilon\,|x|^{-1},
\]
for all $m>\overline{m}$. Consequently,
\[
\dm u(t_m)\dm_{(n,\infty)} \leq c(u_0)\,\varepsilon, \qquad m>\overline{m}.
\]
Passing to the limit as $m\to\infty$, we conclude \eqref{proprieta_L_weak_parta_non_lineare1}.\\
Finally, if $p>2n$, employing again the pointwise estimate,  we arrive at
\[|\overline{u}(t,x)|\leq \sfrac{c(\dm u_0\dm_{\texttt{w},p})}{|x|^{2-\frac{2n}{p}}}\int_0^t\!\!{\tau^{-\frac{n}{p}}(t-\tau)^{-\frac{1}{2}}}\,d\tau\leq {c(u_0)}{|x|^{-2+\frac{2n}{p}}}\,t^{\frac{1}{2}-\frac{n}{p}}\]
that leads to estimate \eqref{proprieta_L_weak_parta_non_lineare}.
\end{proof}

 \begin{proof}[Proof of Theorem\, \ref{esistenza}]
In the hypotheses of Theorem \ref{esistenza}, by virtue of Lemma \ref{lemmaiterativo} and Lemma \ref{convergenza}, we establish a divergence free solution $u(t,x)$ to the integral equation \eqref{Soluzione} such that \eqref{andamento_sol_q} and \eqref{limite_linfinito_soluzione} hold. To the solution $u$, one can associate a pressure term $\pi_u$ given by the representation formula \eqref{Soluzionepres}.
\par
In order to prove the estimate \eqref{stime_pressione_thm_principale} for the pressure term, we argue as follows. \par
Let us fix \( \varepsilon > 0 \) and consider the following system:
\begin{equation}\label{pres1}
\begin{array}{ll}
v_t - \Delta v = -  \nabla \pi_v \,, & \text{in } (\varepsilon,\infty) \times \mathbb{R}^n_{+}, \\
\nabla \cdot v = 0\,, & \text{in } (\varepsilon,\infty) \times \mathbb{R}^n_{+}, \\
v_{| x_n=0} = 0\,, & \\
v(\varepsilon,x) = u(\varepsilon,x)\,, & \text{in } \{\varepsilon\} \times \mathbb{R}^n_{+},
\end{array}
\end{equation}
where $u$ is the solution constructed in Lemma \ref{convergenza}. \par
 Since $u(\varepsilon)\in L^q(\R^n_+)$ for all $q>n$, the well known $L^q$-theory for system \eqref{pres1} ensures the existence of a unique regular solution $(v,\pi_v)$. In particular, the following estimate hold
\begin{equation}\label{stimapresswellknow}
    \dm \n \pi(t)\dm_q \leq (t-\varepsilon)^{-1}\dm u(\varepsilon)\dm_q\,.
\end{equation}
Using \eqref{andamento_sol_q} and \eqref{stimapresswellknow}, for all $\eta>\varepsilon$ and $r>1$,  we have
\begin{equation}\label{stimapres1}
  \int_{\eta}^T\dm \n \pi_v(\tau) \dm_q^r \,d\tau \leq \int_{\eta}^T (\tau-\varepsilon)^{-r}\dm u(\varepsilon)\dm_q^r\,d\tau \leq c_1(r,\varepsilon,\eta, T)\left(   K(\varepsilon,\rho)+  c\dm u_0\dm _{\texttt{w},p}^{1-\gamma}K(\varepsilon,\rho)^{\gamma}\right)^r\,.
\end{equation}
where $K$ is defined in \eqref{krho}.\par
Moreover, since, in particular, for all $\varepsilon >0$, $r>1$ and $q>n$, we have $u\cdot \n u \in L^r(\varepsilon,T;L^q (\R^n_+))$, by the theory developed in \cite{MSolani}, there exists a unique smooth solution $(\bar{w},\pi_{\bar{w}})$ to the problem 
\begin{equation}\label{pres2}
\begin{array}{ll}
\ov{w}_t + \nabla \pi_{\ov{w}} = \Delta \ov{w} -u\cdot \n u\,, & \text{in } (\varepsilon,\infty) \times \mathbb{R}^n_{+}, \\
\nabla \cdot \ov{w} = 0\,, & \text{in } (\varepsilon,\infty) \times \mathbb{R}^n_{+}, \\
\ov{w}_{| x_n=0} = 0\,, & \\
\ov{w}(\varepsilon,x) = 0\,, & \text{in } \{\varepsilon\} \times \mathbb{R}^n_{+},
\end{array}
\end{equation}
such that
\begin{equation*}
 \int_{\eta}^T\dm \n \pi_{\bar{w}}(\tau) \dm_q^r \,d\tau \leq \int_{\eta}^T \dm u\cdot \n u(\tau)\dm_q^r\, d\tau\,,
\end{equation*}
for all $\eta>0$. Applying Holder inequality and using \eqref{andamento_sol_q}, from the last relation we deduce that
\begin{equation}\label{stimapres2}
 \int_{\eta}^T\dm \n \pi_{\bar{w}}(\tau) \dm_q^r \,d\tau \leq c_2(r,\eta, T)\left(   K(\varepsilon,\rho)+  c\dm u_0\dm _{\texttt{w},p}^{1-\gamma}K(\varepsilon,\rho)^{\gamma}\right)^{2r}\,,
\end{equation}
for all $\eta>0$.\par
Now, setting $w:=v+\ov{w}$, and $\pi_w:=\pi_v+\pi_{\ov{w}}$, the couple $(w, \pi_w)$ solves the system 
$$\begin{array}{ll}
w_t  - \Delta w  = - \nabla \pi_{w} -u\cdot \n u\,, & \text{in } (\varepsilon,\infty) \times \mathbb{R}^n_{+}, \\
\nabla \cdot w = 0\,, & \text{in } (\varepsilon,\infty) \times \mathbb{R}^n_{+}, \\
w_{| x_n=0} = 0\,, & \\
w(\varepsilon,x) = u(\varepsilon, x)\,, & \text{in } \{\varepsilon\} \times \mathbb{R}^n_{+}\,.
\end{array}$$
Multiplying the first equation of the last system by $\varphi\in C([0,T);J^{q'}(\R^n_+))\cap L^{q'}(0,T;W^{2,q'}(\R^n_+))$ with $\varphi_t\in L^{q'}(0,T;L^{q'}(\R^n_+))$, since $\displaystyle\lim_{s\to \varepsilon^+}(w(s)-u(\varepsilon), \varphi(s)) =0$, after an integration by parts, we have
\begin{align}
 (w(t),\varphi(t))-(u(\varepsilon),\varphi(\varepsilon))+\int_\varepsilon^t ( u\cdot \n u(\tau), \varphi(\tau)) \,d\tau=0  \label{weeq}
\end{align}
By same arguments, since the solution $u$ to the integral equation \eqref{Soluzione} is also a classical solution (see Remark\,\ref{R-I}), we have
\begin{align}
(u(t),\varphi(t))-(u(\varepsilon),\varphi(\varepsilon))+\int_\varepsilon^t (u\cdot \n u(\tau), \varphi(\tau)) \,d\tau=0 \,.\label{weequ}
\end{align}
Subtracting \eqref{weequ} from \eqref{weeq}, we obtain 
\be \label{diff}
( w(t)-u(t), \varphi(t))=0\,, 
\ee
Now, for $\tau \in (\varepsilon,t)$, let $\varphi(\tau,x)\equiv\Phi(t-\tau+\varepsilon,x)$ be solution backward in time to the Stokes problem \eqref{Stokesf} with $f=0$ corrisponding to the initial datum $\Phi(\varepsilon)\in\mathscr{C}_0(\mathbb{R}^n_+)$. Considering this function and substituting it into formula \eqref{diff}, we have
\[( w(t)-u(t), \Phi(\varepsilon))=0 \,,\]
from which, one deduces that $u\equiv w$. In particular $\pi_u\equiv \pi_w$. Then, by adding estimates
\eqref{stimapres1} and \eqref{stimapres2}, we obtain estimate \eqref{stime_pressione_thm_principale}. The result is completely proved.
\end{proof}
In the following Lemma we prove that any solution, say $v(t,x)$, in our class of existence admits the  decomposition $v(t,x)=\mathcal{G} [v_0] + \mathcal{S} [v\cdot \n v]$. This result allows us to establish uniqueness by focusing only on the coincidence of the nonlinear part, since the linear components already coincide.
\begin{lemma}\label{Dv}
Let $v(t,x)$ a solution to \eqref{problem} corresponding to $v_0$. Suppose that $v(t,x)$ satisfies estimate \eqref{andamento_sol_q} and property \eqref{weak_convergence1}, then $v$ admits the decomposition:
\begin{equation}
v=v^0+\ov v\,,
\end{equation}
where $v^0=\mathcal{G} [v_0]$ and $\ov v=\mathcal{S} [v\cdot \n v] $. 
\end{lemma}
\begin{proof}
Let $v^0:=\mathcal{G} [v_0]$ and set $w:=v-v^0$. We also introduce $\overline{v}$ as the solution to
\begin{equation}\label{Svs}
\begin{cases}
\overline{v}_t-\Delta \overline{v}+\nabla \pi_{\overline{v}}=-v\cdot \nabla v & \text{in } (0,\infty)\times \R^n_+,\\
\nabla\cdot \overline{v}=0 & \text{in } (0,\infty)\times \R^n_+,\\
\overline{v}(0,x)=0 & \text{on } \{0\}\times \R^n_+.
\end{cases}
\end{equation}
It is well known that $\overline{v}$ admits the representation
\[
\overline{v}(t,x)=\int_0^t \nabla G(t-\tau,x-y)\, v\otimes v(\tau,y)\, d\tau.
\]
A direct verification shows that $w$ solves \eqref{Svs} as well. Hence, defining $\overline{w}:=w-\overline{v}$, we obtain
\[
\overline{w}_t-\Delta \overline{w}+\nabla \pi_{\overline{w}}=0,
\]
with $\pi_{\overline{w}}=\pi_v-\pi_{\overline{v}}$.\\
Let $\psi$ be the solution to \rf{Stokesab} with initial datum $\psi_0\in \mathscr C_0(\R^n_+)$ and vanishing coefficients. For fixed $t>0$, define $\widehat{\psi}(\tau,x):=\psi(t-\tau,x)$ for $(\tau,x)\in(0,t)\times\R^n_+$. Then $\widehat{\psi}$ is a backward solution enjoying properties \rf{RHE}--\rf{SP}. In particular, for all $0<s<t$, we have
\begin{equation}\label{CdR}
(\overline{w}(t),\psi_0)=(\overline{w}(s),\psi(t-s)).
\end{equation}
We now estimate the right-hand side of \eqref{CdR}. Writing
\[
(\overline{w}(s),\psi(t-s))=(\overline{w}(s),\psi(t-s)-\psi(t))+(\overline{w}(s),\psi(t)),
\]
we treat the two terms separately.
\\
\noindent\textit{Case $p\in(n,2n)$.}
Using H\"older inequality together with the bounds on $w$ and $\overline{v}$, we obtain
\begin{align*}
|(\overline{w}(t),\psi_0)|
&\leq \dm w(s)\dm_q \dm \psi(t-s)-\psi(t)\dm_{q'}
+ \dm \overline{v}(s)\dm_{\frac{q}{2}} \dm \psi(t-s)-\psi(t)\dm_{\frac{q}{q-2}} \\
&\quad + |(\overline{w}(s),\psi(t))| \\
&\leq c(v_0)c(\psi_0)\, s^{\frac{1}{2}+\frac{n}{2q}}
+ c(\psi_0)\, \dm \overline{v}(s)\dm_{\frac{q}{2}}
+ |(\overline{w}(s),\psi(t))|.
\end{align*}
Passing to the limit as $s\to 0$ and using \eqref{andamento_sol_q} for $w$ together with Lemma~\ref{andamento_a_0_parte_non_lineare} for $\overline{v}$, we deduce
\[
(\overline{w}(t),\psi_0)=0 \quad \text{for all } \psi_0\in \mathscr C_0(\R^n_+).
\]

\medskip
\noindent\textit{Case $p=2n$.}
Proceeding analogously, but using weak Lebesgue norms, we obtain
\begin{align*}
|(\overline{w}(t),\psi_0)|
&\leq \dm w(s)\dm_q \dm \psi(t-s)-\psi(t)\dm_{q'}
+ \dm \overline{v}(s)\dm_{(n,\infty)} \dm \psi(t-s)-\psi(t)\dm_{(n',1)} \\
&\quad + |(\overline{w}(s),\psi(t))| \\
&\leq c(v_0)c(\psi_0)\, s^{\frac{1}{2}+\frac{n}{2q}}
+ \dm \psi_0\dm_{(n',1)}\, \dm \overline{v}(s)\dm_{(n,\infty)}
+ |(\overline{w}(s),\psi(t))|.
\end{align*}
Letting $s\to 0$ and using again the convergence properties of $w$ and $\overline{v}$, we conclude as before.

\medskip
\noindent\textit{Case $p>2n$.}
In this case, we employ the corresponding weak estimates and obtain
\begin{align*}
|(\overline{w}(t),\psi_0)|
&\leq \dm w(s)\dm_q \dm \psi(t-s)-\psi(t)\dm_{q'}
+ \dm \overline{v}(s)\dm_{(\frac{np}{2(p-n)},\infty)}
\dm \psi(t-s)-\psi(t)\dm_{(\frac{np}{p(n-2)+2n},1)} \\
&\quad + |(\overline{w}(s),\psi(t))| \\
&\leq c(v_0)c(\psi_0)\, s^{\frac{1}{2}+\frac{n}{2q}}
+ c\, \dm \overline{v}(s)\dm_{(\frac{np}{2(p-n)},\infty)}
\dm \psi_0\dm_{(\frac{np}{p(n-2)+2n},1)}
+ |(\overline{w}(s),\psi(t))|.
\end{align*}
Passing to the limit as $s\to 0$ yields again $(\overline{w}(t),\psi_0)=0$.
\\
Since $\psi_0$ is arbitrary, we conclude that $\overline{w}\equiv 0$, hence $w=\overline{v}$ and therefore
\[
v=v^0+\overline{v}.
\]
\end{proof}
\begin{proof}[Proof of Theorem \ref{unicita}] The proof of uniqueness relies on a comparison between two solutions \(u(t,x)\) and \(v(t,x)\). Here, \(u\) denotes the solution provided by Theorem~\ref{esistenza}, while \(v\) is assumed to enjoy the same regularity properties. 

According to Lemma~\ref{Dv}, the function \(v\) admits the decomposition \(v = u^0 + \overline{v}\). Depending on the value of \(p\), the remainder term \(\overline{v}\) fulfills the properties stated in \rf{prop_L_parte_non_lineare}, \rf{proprieta_L_weak_parta_non_lineare1}, and \rf{proprieta_L_weak_parta_non_lineare}. 

The uniqueness argument is then carried out via a duality approach, inspired by \cite{F}, with adjustments to fit the present framework.\par
Define \(w(t,x) := v(t,x) - u(t,x)\), which can also be expressed as 
$w(t,x) = \overline{v}(t,x) - \overline{u}(t,x),$
and set \(\pi_w := \pi_v - \pi_u\). 
Taking into account the regularity properties of \(w\) for \(t>0\), it follows that for any \(t > s > 0\), the pair \((w,\pi_w)\) satisfies the following integral formulation:
\begin{equation}\label{WFFW}\ba{ll}
(w(t),\phi(t))\hskip-0.2cm&\displ= (w(s),\phi(s))+ \int_s^t \big[(w,\phi_\tau+\Delta\phi)\VSE\hskip3.2cm+(w\cdot \nabla \phi  ,\overline{v})+(\overline{u} \cdot \nabla \phi,w)+(u^0\cdot \nabla \phi,w)+w\cdot\nabla\phi,u^0\big]d\tau\,,\ea
\end{equation} for all $\phi\in C([0,T);J^{q'}(\R^n_+))\cap L^{q'}(0,T;W^{2,q'}(\R^n_+))$ with $\phi_t\in L^{q'}(0,T;L^{q'}(\R^n_+))$, and $t>s>0$\,. Let \(\psi\) denote the solution to problem \rf{Stokesf} corresponding to the initial datum \(\psi_0 \in \mathscr C_0(\R^n_+)\) and vanishing forcing term \(f \equiv 0\). For any fixed \(t>0\), define the function
\[
\widehat{\psi}(\tau,x) := \psi(t-\tau,x), \qquad (\tau,x) \in (0,t)\times \R^n_+.
\]
With this definition, \(\widehat{\psi}\) solves the associated problem backward in time over \((0,t)\times \R^n_+\), and satisfies the condition \(\widehat{\psi}(t,x) = \psi_0(x)\). Moreover, it fulfills properties \rf{RHE}--\rf{SP}. 

As a consequence, for every \(t > s > 0\), the function \(w\) verifies the following relation:
\begin{equation}\label{ultima}
(w(t),\psi_0)= (w(s),\widehat{\psi}(s))+ \int_s^t [(w\cdot \nabla \widehat{\psi}  ,\overline{v})+(w\cdot\n \widehat \psi, u^0)+(\overline{u} \cdot \nabla \widehat{\psi},w)+(u^0\cdot \nabla \widehat{\psi},w) ]\,d\tau.
\end{equation}
In order to apply Lemma~\ref{andamento_a_0_parte_non_lineare}, we proceed by considering three distinct cases:
\begin{itemize}
\item[1)]Let \(u_0 \in L^p_{\texttt{w}}(\R^n_+)\) with \(p \in (n,2n)\). We select an exponent \(\overline{q} \in (n,2n)\) so as to apply \rf{prop_L_parte_non_lineare} with exponent \(\frac{\overline{q}}{2}\) to both solutions. As a consequence, we obtain in particular
\[
\lim_{s \to 0} \dm  w(s) \dm _{\frac{\overline{q}}{2}} = 0.
\]
Combining this fact with \eqref{ultima} and the estimate for 
\(\dm  \nabla \widehat{\psi} \dm _{\frac{\overline{q}}{\overline{q} - 2}}\), we deduce that
\be\label{ultima2}\ba{ll}
|(w(t),\psi_0)|\hskip-0.2cm&\displ\leq |(w(s),{\psi}(t-s)| + c\sup_{(s,t)}[\tau^{\frac{1}{2}}(\dm \overline{u}(\tau)\dm _{\infty}+ \dm \overline{v}(\tau)\dm _{\infty}+\dm u^0(\tau)\dm _{\infty})] \notag\\
&\displ\hskip2cm\times \sup_{(s,t)}\dm w(\tau)\dm _{\frac{\overline{q}}{2}}\int_s^t \tau^{-\frac{1}{2}}\dm \nabla \widehat{\psi}\dm _{\frac{\overline{q}}{\overline{q}-2}}\,d\tau \notag\\
&\displ\leq \dm w(s)\dm _{\frac{\overline{q}}{2}}\dm {\psi_0}\dm _{\frac{\overline{q}}{\overline{q}-2}} + c\sup_{(s,t)}[\tau^{\frac{1}{2}}(\dm \overline{u}(\tau)\dm _{\infty}+ \dm \overline{v}(\tau)\dm _{\infty}+\dm u^0(\tau)\dm _{\infty})] \notag\\
&\displ\hskip2cm\times\dm \psi_0\dm _{\frac{\overline{q}}{\overline{q}-2}}\sup_{(s,t)}\dm w(\tau)\dm _{\frac{\overline{q}}{2}}\int_s^t \tau^{-\frac{1}{2}}(t-\tau)^{-\frac{1}{2}}\,d\tau\,,
\ea\ee
for all $t \in [0,T(u)) \cap [0,T(v))$. Since $\psi_0$ is arbitrary, letting $s \to 0$, we obtain
\[ \dm w(t)\dm _{\frac{\overline{q}}{2}}\leq  c\sup_{(0,t)}[\tau^{\frac{1}{2}}(\dm \overline{u}(\tau)\dm _{\infty}+ \dm \overline{v}(\tau)\dm _{\infty}+\dm u^0(\tau)\dm _{\infty})]\sup_{(0,t)}\dm w(\tau)\dm _{\frac{\overline{q}}{2}}\]
The first limit condition in \eqref{limite_linfinito_soluzione} immediately implies uniqueness on a time interval of the form \((0,\delta]\). It remains to address the extension of uniqueness to the range \(t \geq \delta\).\par
Since for $t>0$ we have $w$ regular,  we are going to consider  formula  \eqref{ultima} with $s=\delta$. Since $\psi$ enjoys \rf{RHE}, and $\dm w(\delta)\dm _{\frac{\overline{q}}{2}}=0$, for all $r\in[\frac{\ov q}2,\infty)$, we deduce 
\be\label{WB} \dm w(t)\dm _{r}\leq  c\, \delta^{-\frac{1}{2}}\sup_{(\delta,t)}[\tau^{\frac{1}{2}}(\dm \overline{u}(\tau)\dm _{\infty}+ \dm \overline{v}(\tau)\dm _{\infty}+\dm u^0(\tau)\dm _{\infty})]\int_{\delta}^t \dm w(\tau
)\dm _{r}(t-\tau)^{-\frac{1}{2}}\,d\tau\,.\ee
Applying Lemma\,\ref{L-GWSI} with $h(\tau)=\dm w(\tau)\dm _{r}$, we obtain $\dm w(t)\dm _{r}=0$  for all   $t\in[\delta,T)$.

\item[2)] Let \(u_0 \in L^{2n}_{\texttt{w}}(\R^n_+)\). By \rf{proprieta_L_weak_parta_non_lineare1}, we study uniqueness with respect to the quasi-norm \(\dm  \cdot \dm _{(n,\infty)}\). 
Combining \eqref{ultima} with the estimate for 
\(\dm  \nabla \widehat{\psi} \dm _{(n',1)}\) given in \eqref{SPLS}, we obtain:
\be\label{ultima3}\ba{ll}
|(w(t),\psi_0)|\hskip-0.2cm&\displ\leq |(w(s),{\psi}(t-s)| + c\sup_{(s,t)}[\tau^{\frac{1}{2}}(\dm \overline{u}(\tau)\dm _{\infty}+ \dm \overline{v}(\tau)\dm _{\infty}+\dm u^0(\tau)\dm _{\infty})] \notag\\
&\displ\hskip2cm\times \sup_{(s,t)}\dm w(\tau)\dm _{(n,\infty)}\int_s^t \tau^{-\frac{1}{2}}\dm \nabla \widehat{\psi}\dm _{(n',1)}\,d\tau \notag\\
&\displ\leq \dm w(s)\dm _{(n,\infty)}\dm {\psi}_0\dm _{(n',1)} + c\sup_{(s,t)}[\tau^{\frac{1}{2}}(\dm \overline{u}(\tau)\dm _{\infty}+ \dm \overline{v}(\tau)\dm _{\infty}+\dm u^0(\tau)\dm _{\infty})] \notag\\
&\displ\hskip2cm\times\dm \psi_0\dm _{(n',1)}\sup_{(s,t)}\dm w(\tau)\dm _{(n,\infty)}\int_s^t \tau^{-\frac{1}{2}}(t-\tau)^{-\frac{1}{2}}\,d\tau\,,
\ea\ee
for all $t \in [0,T(u)) \cap [0,T(v))$. The arbitrariness of $\psi_0$ and Lemma\,\ref{HDLS} allow us to conclude that
\begin{align*}
\dm w(t)\dm _{(n,\infty)}\leq c\sup_{(0,t)}[\tau^{\frac{1}{2}}(\dm \overline{u}(\tau)\dm _{\infty}+ \dm \overline{v}(\tau)\dm _{\infty}+\dm u^0(\tau)\dm _{\infty})]\sup_{(0,t)}\dm w(\tau)\dm _{(n,\infty)}.
\end{align*}
The validity of the limit condition in \eqref{limite_linfinito_soluzione} guarantees uniqueness on an interval of the form \((0,\delta]\). Furthermore, for \(t>\delta\) we proceed as in the argument already used in \rf{WB}. In this situation, we observe that at time \(s=\delta\), and without loss of generality, one has \(\dm  w(\delta) \dm _{r} = 0\). Consequently, for any \(r \in (p,\infty)\), we define
\be \dm  w(t)\dm _{r} \leq  c\, \delta^{-\frac{1}{2}}\sup_{(\delta,t)}[\tau^{\frac{1}{2}}(\dm \overline{u}(\tau)\dm _{\infty}+ \dm \overline{v}(\tau)\dm _{\infty}+\dm u^0(\tau)\dm _{\infty})]\int_{\delta}^t \dm  w(\tau)\dm _{r} (t-\tau)^{-\frac{1}{2}}\,d\tau.\ee
Applying Lemma\,\ref{L-GWSI} with $h(\tau)=\dm w(\tau)\dm _{r}$, we obtain $\dm  w(t)\dm _{r}=0$\, for all  $t\in[\delta,T)$.
\item[3)]Let $u_0 \in L^p_{\texttt{w}}(\R^n_+)$ with $p>2n$. We consider the norm $\dm \cdot\dm _{(\frac{np}{2(p-n)},\infty)}$.
From \eqref{ultima},  we get:
\begin{align*}
|(w(t),\psi_0)|&\leq |(w(s),{\psi}(t-s)| + c\sup_{(s,t)}[\tau^{\frac{1}{2}}(\dm \overline{u}(\tau)\dm _{\infty}+ \dm \overline{v}(\tau)\dm _{\infty}+\dm u^0(\tau)\dm _{\infty})]\times \notag\\
&\hskip2cm\times \sup_{(s,t)}\dm w(\tau)\dm _{(\frac{np}{2(p-n)},\infty)}\int_s^t \tau^{-\frac{1}{2}}\dm \nabla \widehat{\psi}\dm _{(\frac{np}{np-2(p-n)},1)}d\tau \notag\\
&\leq \dm w(s)\dm _{(\frac{np}{2(p-n)},\infty)}\dm \widehat{\psi}(s)\dm _{(\frac{np}{np-2(p-n)},1)} + c\sup_{(s,t)}[\tau^{\frac{1}{2}}(\dm \overline{u}(\tau)\dm _{\infty}+ \dm \overline{v}(\tau)\dm _{\infty}+\dm u^0(\tau)\dm _{\infty})]\times \notag\\
&\hskip2cm\times\dm \psi_0\dm _{(\frac{np}{np-2(p-n)},1)}\sup_{(s,t)}\dm w(\tau)\dm _{(\frac{np}{2(p-n)},\infty)}\int_s^t \tau^{-\frac{1}{2}}(t-\tau)^{-\frac{1}{2}}\,d\tau\,.
\end{align*}
Taking into account property \eqref{proprieta_L_weak_parta_non_lineare} and repeating the same line of reasoning as in cases 1)--2), we conclude that
\[
\dm  w(t) \dm _{\left(\frac{np}{2(p-n)},\infty\right)} = 0 \qquad \text{for all } t \in (0,T).
\]
\end{itemize}

This completes the proof.
\end{proof}

 {\bf Acknowledgment} -
We would like to thank Professor P. Maremonti for his valuable supervision. \par
The paper is performed under the auspices of GNFM-INdAM.
\vskip0.1cm\noindent
 {\bf Declarations}
\vskip0.1cm\noindent
 {\bf Conflict of interest} - The authors have no conflicts of interest to declare that are relevant to the content of this article.


\end{document}